\documentclass[a4paper,12pt]{article}

 \usepackage[utf8]{inputenc}
 \usepackage[T1]{fontenc}
 \usepackage{textcomp}
 \usepackage[normalem]{ulem}
 \usepackage{amsmath}
 \usepackage{amssymb}
 \usepackage{amsthm}
 \usepackage{vmargin} 
 \usepackage{graphicx}
 \usepackage{color}
 \usepackage{epsfig}
 \usepackage[only,llbracket,rrbracket]{stmaryrd}
 \usepackage{hyperref}

\newtheorem{remark}{Remark}[section]
\newtheorem{deft}{Definition}[section]
\newtheorem{thm}{Theorem}
\newtheorem{prop}{Proposition}[section]
\newtheorem{lemma}{Lemma}

\title{Dynamics in a kinetic model of oriented particles with phase transition}
\author{Amic Frouvelle\thanks{Institut de Mathématiques de Toulouse, CNRS -- UMR 5219, Université de Toulouse, F-31062 Toulouse, France, \href{mailto:amic.frouvelle@math.univ-toulouse.fr}{amic.frouvelle@math.univ-toulouse.fr} } 
\and Jian-Guo Liu\thanks{Department of Physics and Department of Mathematics, 
Duke University, Durham, NC 27707, USA, \href{mailto:Jian-Guo.Liu@duke.edu}{Jian-Guo.Liu@duke.edu}}}
\date{}

\begin{document}
\maketitle
 
\begin{abstract}
Motivated by a phenomenon of phase transition in a model of alignment of self-propelled particles, we obtain a kinetic mean-field equation which is nothing else than the Doi equation (also called Smoluchowski equation) with dipolar potential.

In a self-contained article, using only basic tools, we analyze the dynamics of this equation in any dimension. 
We first prove global well-posedness of this equation, starting with an initial condition in any Sobolev space.
We then compute all possible steady-states. 
There is a threshold for the noise parameter: over this threshold, the only equilibrium is the uniform distribution, and under this threshold, there is also a family of non-isotropic equilibria.

We give a rigorous prove of convergence of the solution to a steady-state as time goes to infinity. 
In particular we show that in the supercritical case, the only initial conditions leading to the uniform distribution in large time are those with vanishing momentum. 
For any positive value of the noise parameter, and any initial condition, we give rates of convergence towards equilibrium, exponentially for both supercritical and subcritical cases and algebraically for the critical case. 

\end{abstract}


\newpage

\section{Introduction}
Phase transition and large time behavior of large interacting oriented/rod-like particle systems and their mean field limits have shown to be interesting in many physical and biological complex systems. 
Examples are: paramagnetism to ferromagnetism phase transition near Curie temperature, nematic phase transition in liquid crystal or rod-shaped polymers, emerging of flocking dynamics near critical mass of self-propelled particles, etc. 

The dynamics on orientation for self-propelled particles proposed by Vicsek \emph{et al}~\cite{vicsek1995novel} to describe, for instance, fish schooling or bird flocking, present such a behavior in numerical simulations.
As the density increases (or as the noise decreases) and reaches a threshold one can observe strong correlations between the orientations of particles. 
The model is discrete in time and particles move at constant speed following their orientation. At each time step, the orientation of each particle is updated, replaced by the mean orientation of its neighbors, plus a noise term.

A way to provide a time-continuous version of this dynamical system, which allows to take a mean-field limit (and even a macroscopic limit), has been proposed by Degond and Motsch \cite{degond2008continuum}. 
Instead of replacing the orientation at the next time step, they introduce a parameter playing the role of a rate of relaxation towards this mean orientation. 
Unfortunately the mean-field limit of this model does not present phase transition.
In \cite{frouvelle2011continuous}, the first author of the present paper proved the robustness of the behavior of this model when this rate of relaxation depends on a local density. 
In particular, phase transition is still absent. 
However, when this parameter is set to be proportional to the local momentum of the neighboring particles, we will see that the model present a phenomenon of phase transition.
This phenomenon occurs on the orientation dynamics, so we will only consider here the spatial homogeneous dynamics. 
A current joint work with Pierre Degond is dedicated to the study of this model when we take in account the space variable. It is left aside in the present paper, and a specific paper \cite{degond2011macroscopic} is in progress on this subject. 

The particular model is described as follows: we have $N$ oriented particles, described by vectors~$\omega_1,\dots,\omega_N$ belonging to $\mathbb{S}$, the unit sphere of $\mathbb{R}^n$, and satisfying the following system of coupled stochastic differential equations, for $k\in \llbracket 1,N\rrbracket $:
\begin{align}
\mathrm d\omega_k &= (\mathrm{Id} - \omega_k \otimes \omega_k) (J_k \, \mathrm dt + \sqrt{2\sigma} \, \mathrm dB^k_t), 
\label{Part_Dyn_orient_sto} \\
J_k &= \frac1N\sum_{j=1}^N \omega_j. 
\label{Part_Dyn_J} 
\end{align}
The term $(\mathrm{Id} - \omega_k \otimes \omega_k)$ denotes the projection on the hyperplane orthogonal to $\omega_k$, and constrains the norm of $\omega_k$ to be constant. 
The terms $B^k_t$ stand for $N$ independent standard Brownian motions on $\mathbb{R}^n$, and then the stochastic term $(\mathrm{Id} - \omega_k \otimes \omega_k)\mathrm dB^k_t$ represents the contribution of a Brownian motion on the sphere $\mathbb{S}$ to the model. For more details on how to define Brownian motion on a Riemannian manifold, see \cite{hsu2002stochastic}. 

Without this stochastic term, equation \eqref{Part_Dyn_orient_sto} can be written
\begin{equation*}
\dot{\omega_k}=\nabla_\omega (\omega \cdot J_k)|_{\omega =\omega_k},
\end{equation*}
where $\nabla_\omega $ is the tangential gradient on the sphere (see the beginning of Section~\ref{preliminaries} for some useful formulas on the unit sphere). So the model can be understood as a relaxation towards a unit vector in the direction of $J_k$, subjected to a Brownian motion on the sphere with intensity $\sqrt{2\sigma}$.
The only difference with the model proposed in \cite{degond2008continuum} (in the spatial homogeneous case) is that $J_k$ is there replaced by $\nu \Omega_k$, where $\Omega_k$ is the unit vector in the direction of $J_k$ and the frequency of relaxation $\nu $ is constant (or dependent on the local density in \cite{frouvelle2011continuous}). 
One point to emphasize is that, in that model, the interaction cannot be seen as a sum of binary interactions, contrary to the model presented here.
Here the mean momentum $J_k$ does not depend on the index $k$ (but this is not true in the inhomogeneous case, where the mean is taken among the neighboring particles).

To simplify notations, we work with the uniform measure of total mass $1$ on the sphere $\mathbb{S}$.
We denote by $f^N:\mathbb{R}_+\times \mathbb{S} \to \mathbb{R}_+$ the probability density function (depending on time) associated to the position of one particle. Then, as the number~$N$ of particles tends to infinity, $f^N$ tends to a probability density function $f$ satisfying
\begin{equation}
\partial_t f = Q(f), 
\label{Doi_eq} 
\end{equation}
with
\begin{align}
Q(f) &= - \nabla_\omega \cdot ((\mathrm{Id} - \omega \otimes \omega) J[f] f) +\sigma \Delta_\omega f, 
\label{Q_def} \\
J[f] &= \int_{\mathbb{S}} \omega \, f(.,\omega) \, \mathrm d\omega. 
\label{J_def} 
\end{align}
In the model of \cite{degond2008continuum}, $J[f]$ is just replaced in \eqref{Q_def} by $\nu \,\Omega [f]$, where $\Omega [f]$ is the unit vector in the direction of $J[f]$. 

The first term of $Q(f)$ can be formally derived using a direct computation with the empirical distribution of particles.
And the diffusion part comes from Itô’s formula. 
A rigorous derivation of this mean-field limit is outside the scope of the present paper, and is linked with the so-called “propagation of chaos” property.
We refer to \cite{sznitman1991topics} for an introduction to this notion. 
The laboratory example given in this reference is the original model of McKean \cite{mckean1967propagation} which is a more general version of our system in~$\mathbb{R}^n$ instead of $\mathbb{S}$ (in that case, equation \eqref{Doi_eq} is called McKean-Vlasov equation).
The main point is to adapt the theory in the framework of stochastic analysis on Riemannian manifolds.

Notice that equation \eqref{Doi_eq} can be written in the form
\begin{equation*}
\partial_t f = \nabla \cdot (f\nabla \Psi) + \sigma \Delta f,
\end{equation*}
with
\begin{equation*}
\Psi (\omega,t)=-\omega \cdot J(t)=\int_{\mathbb{S}} K(\omega,\bar{\omega}) \, f(t,\bar{\omega}) \, \mathrm d\bar{\omega}.
\end{equation*}

This equation is known as Doi equation (or Doi-Onsager, Smoluchowski, or even nonlinear Fokker-Planck equation) and was introduced by Doi \cite{doi1981molecular} as gradient flow equation for the Onsager free energy functional:
\begin{equation}
\mathcal F(f)=\sigma \int_\mathbb{S} f(.,\omega) \ln f(.,\omega) \mathrm d\omega + \tfrac12\int_{\mathbb{S}\times \mathbb{S}} K(\omega,\bar{\omega}) f(.,\omega)\, f(.,\bar{\omega}) \,\mathrm d\omega \mathrm d\bar{\omega}.
\label{Onsager_free_energy}
\end{equation}
This functional was proposed by Onsager \cite{onsager1949effects} to describe the equilibrium states of suspensions of rod-like polymers. They are given by the critical points of this functional. 

Defining the chemical potential $\mu $ as the first order variation of $\mathcal F(f)$ under the constraint $\int_{\mathbb{S}}f=1$, we get $\mu =\sigma \ln f + \Psi $, and the Doi equation becomes 
\begin{equation*}
\partial_t f = \nabla \cdot (f \nabla \mu).
\end{equation*}

In the original work of Onsager, the kernel has the form $K(\omega,\bar{\omega})=|\omega \times \bar{\omega}|$, but there is another form, introduced later by Maier and Saupe \cite{maier1958eine}, which leads to similar quantitative results: $K(\omega,\bar{\omega})=-(\omega \cdot \bar{\omega})^2$. 
In our case, the potential given by $K(\omega,\bar{\omega})=-\omega \cdot \bar{\omega}$ is called the dipolar potential. 
This is a case where the arrow of the orientational direction has to be taken in account.

One of the interesting behavior of the Doi-Onsager equation is the phase transition bifurcation. 
This is indeed easy to see (here with the dipolar potential) from the following linearization around the uniform distribution: if $f$ is a probability density function, solution of \eqref{Doi_eq}, we write $f=1+g$, so $\int_\mathbb{S} g\,\mathrm d\omega =0$ and we can get the equation for $g$. 
We multiply the equation by $\omega $ and integrate, using the formula~$\int_\mathbb{S} \omega \otimes \omega \, \mathrm d\omega = \frac1n\,\mathrm{Id}$ (this is a matrix with trace one and commuting with any rotation) and the tools in the beginning of Section \ref{preliminaries}. 
We get the linearized equation for $g$ and $J[g]$:
\begin{gather*}
 \partial_t g = \sigma \Delta_\omega g + (n-1)\,\omega \cdot J[g] + O(g^2),\\
 \frac{\mathrm d}{\mathrm dt} J[g] = (n-1)\left(\dfrac{1}n-\sigma \right)J[g] + O(g^2).
\end{gather*}

Therefore if we take the linear part of this system, we can solve the second equation directly, and the first one becomes the heat equation with a known source term.
Finally, around the constant state, the linearized Doi equation is stable if $\sigma \geqslant \frac1n$, and unstable if $\sigma <\frac1n$. 
We expect to find another kind of equilibrium in this regime. The work has been done in \cite{fatkullin2005critical} for the dimension $n=3$, the distribution obtained is called Fisher-Von Mises distribution \cite{watson1982distributions}. As far as we know, this is the only work dealing with the dipolar potential alone.

A lot of work has been done to study the equilibrium states for the Maier-Saupe potential, and in particular to show the axial symmetry of these steady states.
A complete classification has been achieved for the two and three-dimensional cases in~\cite{liu2005axial} (see also~\cite{zhang2007stable}, including the analysis of stability under a weak external shear flow). 
The interesting behavior, besides the phase transition, is the hysteresis phenomenon: before a first threshold, only the anisotropic equilibrium is stable, then both anisotropic and uniform equilibria are stable, and after a second threshold, the only equilibrium is the uniform distribution.
In the case of a coupling between the Maier-Saupe and the dipolar, it is shown in \cite{zhou2007characterization} that the only stable equilibrium states are axially symmetric.
To our knowledge, less work has been done to study the dynamics of the Doi-Onsager equation, in particular the rate at which the solution converges to a steady-state.

The purpose of this paper is to give a rigorous proof of the phase transition in any dimension for the dipolar potential, and study the large time dynamics and the convergence rates towards equilibrium states.

In Section \ref{general_results}, we give some general results concerning equation \eqref{Doi_eq}. 
We provide a self-contained proof for existence and uniqueness of a solution with initial nonnegative condition in any Sobolev space. 
We show that the solution is instantaneously positive and in any Sobolev space (and actually analytic in the space variable), and we obtain uniform bounds in time for each Sobolev norm.

In Section \ref{using_free_energy}, we use the Onsager free energy (decreasing in time) to analyze the general behavior of the solution as time goes to infinity. 
We prove a kind of LaSalle principle, implying that the solution converges, in the $\omega $-limit sense, to a given set of equilibria. 
We determine all the steady states, and see that the value~$\frac1n$ is indeed a threshold for the noise parameter $\sigma $. 
Over this threshold, the only equilibrium is the uniform distribution. 
When $\sigma <\frac1n$, two kinds of equilibria exist: the uniform distribution, and a family of non-isotropic distributions (called Fischer-Von Mises distributions), with a concentration parameter $\kappa $ depending on $\sigma $.
 
Finally, in Section \ref{convergence}, we show that the solution converges strongly to a given equilibrium. We first obtain a new conservation relation, which plays the role of an entropy when $\sigma \geqslant \frac1n$, and shows a global convergence to the uniform distribution with rate proportional to $\sigma -\frac1n$. Then we prove that, in the supercritical case $\sigma <\frac1n$, the solution converges to a non-isotropic equilibrium if and only if the initial drift velocity $|J[f_0]|$ is non-zero (if it is zero, the equation reduces to the heat equation, and the solution converges exponentially fast to the uniform distribution). We prove in that case that the convergence to this steady-state is exponential in time, and we give the asymptotic rate of convergence. Finally, in the critical case $\sigma =\frac1n$, we show that the speed of convergence to the uniform distribution is algebraic (more precisely the decay in any Sobolev norm is at least $\frac{C}{\sqrt t}$).

\section{General results}
\label{general_results}
\subsection{Preliminaries: some results on the unit sphere}
\label{preliminaries}

This subsection consists essentially in a main lemma, allowing to perform some estimates on the norm of integrals of the form $\int_\mathbb{S} g\nabla_\omega h $, where $h$ and $g$ are real functions with mean zero.

But let us start by some useful formulas.

For $V$ a constant vector in $\mathbb{R}^n$, we have:
\begin{gather*}
 \nabla_\omega (\omega \cdot V) = (\mathrm{Id} - \omega \otimes \omega) V\\
 \nabla_\omega \cdot ((\mathrm{Id} - \omega \otimes \omega)V)= -(n-1)\,\omega \cdot V,
\end{gather*}
where $\nabla_\omega $ (resp. $\nabla_\omega \cdot $) stands for the tangential gradient (resp. the divergence) on the unit sphere. When no confusion is possible, we will just use the notation $\nabla $.

Then, taking the dot product with a given tangent vector field $A$ or multiplying by a regular function $f$ and integrating by parts, we get
\begin{gather*}
 \int_\mathbb{S} \omega \,\nabla_\omega \cdot A(\omega) \mathrm d\omega = -\int_\mathbb{S} A(\omega) \mathrm d\omega \\
 \int_\mathbb{S} \nabla_\omega f \mathrm d\omega = (n-1)\int_\mathbb{S} \omega f \mathrm d\omega \label{grad_flux}.
\end{gather*}

We then introduce some notations.
We denote by $\dot{H}^s(\mathbb{S})$ the subspace composed of mean zero functions of the Sobolev space $H^s(\mathbb{S})$.
This is a Hilbert space, associated to the inner product $\langle g,h\rangle_{\dot{H}^s}^2= \langle (-\Delta)^sg,h\rangle $, where $\Delta $ is the Laplace-Beltrami operator on the sphere. 
This has also a sense for any $s\in \mathbb{R}$ by spectral decomposition of this operator. We will denote by $\|\cdot \|_{\dot{H}^s}$ the norm on this Hilbert space.

We then define the so-called conformal Laplacian $\widetilde{\Delta}_{n-1}$ on the sphere (see \cite{beckner1993sharp}) which plays a role in some Sobolev inequalities. 
This is a positive definite operator (pseudodifferential operator of degree $n-1$, mapping continuously $\dot{H}^s(\mathbb{S})$ into~$\dot{H}^{s-n+1}(\mathbb{S})$, which is a differential operator when $n$ is odd) given by 
\begin{equation}
\label{conformal_laplacian}
\widetilde{\Delta}_{n-1}=
\begin{cases}
\displaystyle\prod_{ 0\leqslant j\leqslant \frac{n-3}2}\left(-\Delta +j(n-j-2)\right) & \text{ for } n \text{ odd,}\\
\left(-\Delta +(\frac n2-1)^2\right)^{\frac12}\displaystyle\prod_{ 0\leqslant j\leqslant \frac n2-2}\left(-\Delta +j(n-j-2)\right) & \text{ for } n \text{ even.}
\end{cases}
\end{equation}
Equivalently, it can be also defined by 
\begin{equation}
\label{spectral_conformal_laplacian}
\widetilde{\Delta}_{n-1}\, Y_\ell = \ell (\ell +1)\dots (\ell +n-2)Y_\ell \text{ for any spherical harmonic }Y_\ell \text{ of degree }\ell.
\end{equation}

Here is the main lemma.

\begin{lemma}
\label{main_lemma}
Estimates on the sphere.
\begin{enumerate}
\item If $h$ in $\dot{H}^{-s+1}(\mathbb{S})$ and $g$ in $\dot{H}^{s}(\mathbb{S})$, the following integral is well defined and we have
\begin{equation}
\left|\int_\mathbb{S} g\nabla h\right|\leqslant C\|g\|_{\dot{H}^s}\|h\|_{\dot{H}^{-s+1}}
\end{equation}
where the constant $C$ depends only on $s$ and $n$.
\item
We have the following estimation, for any $g \in \dot{H}^{s+1}(\mathbb{S})$:
\begin{equation}
\left|\int_{\mathbb{S}}g \nabla (-\Delta)^sg\right|\leqslant C\|g\|^2_{\dot{H}^{s}},
\label{commutator_estimate}
\end{equation}
where the constant $C$ depends only on $s$ and $n$.
\item We have the following identity, for any $g \in \dot{H}^{-\frac{n-3}2}$:
\begin{equation}
\int_{\mathbb{S}}g \nabla \widetilde{\Delta}_{n-1}^{-1}g=0
\label{exact_commutator}
\end{equation}
\end{enumerate}
\end{lemma}

Let us make some remarks on these statements. 
The first one is just expressing the fact that the gradient operator (or more precisely any of its component $e\cdot \nabla $ for a given unit vector $e$) is well defined as an operator sending $\dot{H}^{-s+1}(\mathbb{S})$ continuously into $\dot{H}^{-s}(\mathbb{S})$ for any $s$.

The second one is actually a commutator estimate. 
It is equivalent to the fact that for any given unit vector $e$, and for any $g,h \in \dot{H}^{s+1}$ we have 
\begin{equation*}
\left|\int_{\mathbb{S}}g e\cdot \nabla (-\Delta)^sh+h e\cdot \nabla (-\Delta)^sg\right|\leqslant \widetilde C\|g\|_{\dot{H}^{s}}\|h\|_{\dot{H}^{s}}.
\end{equation*} 
Defining the operator $F$ by 
\begin{equation*}
Fg=e\cdot \nabla (-\Delta)^sg-(-\Delta)^s\nabla \cdot ((\mathrm{Id} - \omega \otimes \omega)eg)
\end{equation*} 
and integrating by parts, this inequality becomes $\left|\int_{\mathbb{S}}h\,Fg\right|\leqslant \widetilde C\|g\|_{\dot{H}^{s}}\|h\|_{\dot{H}^{s}}$. 
In other words, $F$ sends $\dot{H}^{s}(\mathbb{S})$ continuously into $\dot{H}^{-s}(\mathbb{S})$ for any $s$. 

So since $F=[e\cdot \nabla,(-\Delta)^s]+(n-1)(-\Delta)^se\cdot \omega $, this second statement \eqref{commutator_estimate} expresses that the commutator $[\nabla,(-\Delta)^s]$ is an operator of degree $2s$.

With the same point of view, the last equality \eqref{exact_commutator} gives an exact computation of the commutator of the gradient and the inverse of conformal Laplacian.

This is just saying that $[\nabla,\widetilde{\Delta}_{n-1}^{-1}]=-(n-1)\widetilde{\Delta}_{n-1}^{-1}\omega $, or, multiplying left and right by~$\widetilde{\Delta}_{n-1}$, that $[\nabla,\widetilde{\Delta}_{n-1}]=(n-1)\omega \widetilde{\Delta}_{n-1}$.

The proof of this lemma relies on some computations on spherical harmonics, and is given in Appendix \ref{appendix_harmonics}.

\subsection{Existence, uniqueness, positivity, regularity.}

We present here a self-contained proof of well-posedness of the problem \eqref{Doi_eq}, working in any Sobolev space for the initial condition. 
Some analogous claims are given in~\cite{constantin2005dissipativity}, without proof, starting for a continuous nonnegative function. 
They are based on arguments of \cite{cao2000gevrey}, stating that the Galerkin method based on spherical harmonics converges (exponentially fast) to the unique solution. 
They are weaker with respect to the initial conditions and the positivity, but stronger for the regularity of the solution (analytic in space). As a remark we will give the same regularity results, and prove it in Appendix \ref{analyticity}.

\begin{deft}\label{definition_weak_solution}
Weak solution.

For $T>0$, the function $f \in L^2((0,T),H^{s+1}(\mathbb{S}))\cap H^1((0,T),H^{s-1}(\mathbb{S}))$ is said to be a weak solution of~\eqref{Doi_eq} if for almost all $t \in [0,T]$, we have for all $h \in H^{-s+1}(\mathbb{S})$
\begin{equation}
 \langle \partial_t f,h\rangle =-\sigma \langle \nabla_\omega f,\nabla_\omega h\rangle + \langle f,J[f]\cdot \nabla_\omega h\rangle,
\label{weak_solution}
\end{equation}
where $\langle \cdot,\cdot \rangle $ is the usual duality product for distributions on the sphere $\mathbb{S}$.
\end{deft}
Since it is sometimes more convenient to work with mean zero functions (in order to use the main lemma of the previous subsection), we reformulate this problem in another framework. 
We set $f=1+g$ so that $f$ is a weak solution if and only if~$g \in L^2((0,T),\dot{H}^{s+1}(\mathbb{S}))\cap H^1((0,T),\dot H^{s-1}(\mathbb{S}))$ with, for almost all $t \in [0,T]$, and for all~$h \in \dot H^{-s+1}(\mathbb{S})$, 
\begin{equation}
 \langle \partial_t g,h\rangle =-\sigma \langle \nabla_\omega g,\nabla_\omega h\rangle + (n-1)J[g]\cdot J[h] + \langle g,J[g]\cdot \nabla_\omega h\rangle.
\label{weak_solution_meanzero}
\end{equation}

That makes sense to look for a weak solution with prescribed initial condition in~$H^s$, since it always belongs to~$C([0,T],H^s(\mathbb{S}))$, as stated by the following proposition.

\begin{prop}
\label{embedding_IC}
If $g \in L^2((0,T),\dot{H}^{s+1}(\mathbb{S}))\cap H^1((0,T),\dot H^{s-1}(\mathbb{S}))$, then, up to redefining it on a set of measure zero, it belongs to~$C([0,T],\dot H^s(\mathbb{S}))$, and we have
\begin{equation*}
\max_{[0,T]}\|u(t)\|^2_{\dot H^s}\leqslant C\int_0^T\|u\|^2_{\dot H^{s+1}}+\|\partial_tu\|^2_{\dot H^{s-1}},
\end{equation*}
where the constant $C$ depends only on $T$.
\end{prop}
The proof in the case $s=0$ is the same as in \cite{evans1998partial}, Thm 3, §5.9.2. To do the general case, we apply the result to $(-\Delta)^{\frac s2}g$. 

\begin{thm}
\label{generalresults}
Given an initial probability measure $f_0$ in $H^s(\mathbb{S})$, there exists a unique weak solution $f$ of \eqref{Doi_eq} such that $f(0)=f_0$. 
This solution is global in time (the definition above is valid for any time $T>0$). Moreover, $f\in C^\infty ((0,+\infty)\times \mathbb{S})$, with~$f(t,\omega)>0$ for all positive $t$. 

We also have the following instantaneous regularity and uniform boundedness estimates (for $m\in \mathbb{N}$, the constant $C$ depending only on $\sigma,m,s$), for all $t>0$:
\begin{equation*}
\|f(t)\|^2_{H^{s+m}}\leqslant C\left(1+\frac1{t^m}\right)\|f_0\|^2_{H^{s}}.
\end{equation*}

\end{thm}

The proof consists in several steps, which we will treat as propositions.
We first use a Galerkin method to prove existence on a small interval. We then show the continuity with respect to initial conditions on this interval (and so the uniqueness). 
Next, we prove the positivity of $1+g$ for regular solutions. 
This gives us a better estimate of $J[g]$. 
Repeating the procedure on the following small interval, and so on, we can show that this extends to any $t>0$. 
Regularizing the initial condition give then global existence in any case.

We finally obtain the instantaneous regularity and boundary estimates by decomposing the solution between low and high modes.

For the proof of all propositions, we will denote by $C_0,C_1,\dots $ some positive constants which depends only on $s$ and $\sigma $.
We will also fix one parameter $K>0$ (which will be a bound on the norm of initial condition), and denote by $M_0,M_1,\dots $ some positive constants which depends only on $s$ and $\sigma $, and $K$.

\begin{prop}Existence: Galerkin method.

We set 
\begin{equation}
T=\frac1{C_1}\ln\left(1+\frac1{1+2C_2K}\right), 
\label{def_small_time}
\end{equation}
where the constant $C_1$ and $C_2$ will be defined later.

If $\|g_0\|_{\dot H^s}\leqslant K$, then we have existence of a weak solution on $[0,T]$ satisfying~\eqref{weak_solution_meanzero}, uniformly bounded in $L^2((0,T),\dot{H}^{s+1}(\mathbb{S}))\cap H^1((0,T),\dot H^{s-1}(\mathbb{S}))$ by a constant $M_1$.
\end{prop}

\begin{proof}
We denote by $P_N$ the space spanned by the first $N$ (non-constant) eigenvectors of the Laplace-Beltrami operator. 
This is a finite dimensional vector space, included in $\dot{H}^{p}(\mathbb{S})$ for all $p$, and containing the functions of the form $\omega \mapsto V\cdot \omega $ (see Appendix \ref{appendix_harmonics} for more details).
 
Let $g^N \in C^1(I,P_N)$ be the unique solution of the following Cauchy problem, defined on a maximal interval $I\subset \mathbb{R}_+$ (``non-linear'' ODE on a finite dimensional space):
\begin{equation*}
\begin{cases}
 \frac{\mathrm d}{\mathrm dt}g^N = \Pi_N(\sigma \Delta_\omega g^N + (n-1)(1+g^N)\,\omega \cdot J[g^N] - J[g^N]\cdot \nabla_\omega g^N),\\ 
 g^N(0)=\Pi_N(g_0),
\end{cases}
\end{equation*}
where $\Pi_N$ is the orthogonal projection on $P_N$. 
The first equation is equivalent to the fact that for any $h \in P_N$, we have
\begin{equation}
 \frac{\mathrm d}{\mathrm dt}\langle g^N,h\rangle =-\sigma \langle \nabla_\omega g^N,\nabla_\omega h\rangle + (n-1)J[g^N]\cdot J[h] + \langle g^N,J[g^N]\cdot \nabla_\omega h\rangle.
\label{Galerkin}
\end{equation}
The goal is to prove that~$[0, T]\subset I$ and that there exists an extracted sequence~$N_k$ such that, as $k\rightarrow \infty $,
\begin{itemize}
\item $g^{N_k}$ converges weakly in~$L^2((0,T),\dot{H}^{s+1}(\mathbb{S}))$ to a function $g$,
\item $\partial_t g^{N_k}$ converges weakly to~$\partial_t g$ in~$L^2((0,T),\dot{H}^{s}(\mathbb{S}))$,
\item $J[g^{N_k}]\rightarrow J[g]$ uniformly.
\end{itemize}
We have that $(-\Delta)^sg^N \in P_N$, so we can take it for $h$, put it in \eqref{Galerkin} and use the second part of Lemma \ref{main_lemma} to get:
\begin{align}
\frac12\frac{\mathrm d}{\mathrm dt} \|g^N\|^2_{\dot{H}^{s}}+\sigma \|g^N\|^2_{\dot{H}^{s+1}}&\leqslant C_0|J[g^N]|\|g^N\|^2_{\dot{H}^{s}}+(n-1)^s|J[g^N]|^2\label{estimate_galerkin}\\
&\leqslant C_1\|g^N\|^2_{\dot{H}^{s}}(1+C_2\|g^N\|_{\dot{H}^{s}})\label{estimate_galerkin_rough}.
\end{align}
Indeed, any component of $\omega $ belongs to any $\dot{H}^{-s}$, then $J[g^N]=\langle \omega,g^N\rangle $ is controlled by any $\dot{H}^{s}$ norm of $g^N$. 

Solving this inequality, we obtain for $0\leqslant t<C_1^{-1}\ln(1+(C_2\|\Pi^N(g_0)\|_{\dot H^s})^{-1})$, 
\begin{align}
\|g^N\|_{\dot{H}^{s}}\leqslant \frac{\|\Pi^N(g_0)\|_{\dot H^s}}{e^{-C_1t}-C_2\|\Pi^N(g_0)\|_{\dot H^s}(1-e^{-C_1t})}\label{sol_ineq_galerkin}.
\end{align}
Then we have $\|g^N(t)\|_{\dot{H}^{s}}\leqslant 2\|g_0\|_{\dot H^s}$ for all $t$ in $[0,T]$.
There is no finite-time blow up in $[0,T]$, then the ODE $\eqref{Galerkin}$ has a solution on~$[0,T]$, for any $N \in \mathbb{N}$.

Now we denote by $M_0$ a bound for $|J[g^N]|$ on $[0,T]$. The inequality \eqref{estimate_galerkin} gives
\begin{equation*}
\frac12\frac{\mathrm d}{\mathrm dt} \|g^N\|^2_{\dot{H}^{s}}+\sigma \|g^N\|^2_{\dot{H}^{s+1}}\leqslant (1+M_0)C_3\|g\|^2_{\dot{H}^{s}}.
\end{equation*}
Solving this inequality, we get for $t\in [0,T]$
\begin{equation*}
\|g^N\|^2_{\dot{H}^s} + \sigma \int_0^{T}\|g^N\|^2_{\dot{H}^{s+1}} \leqslant \|g_0\|^2_{\dot{H}^s}e^{(1+M_0) C_3 T}.
\end{equation*}
We then use the ODE \eqref{Galerkin} and this estimate to control the derivative of $g$. 

Taking $h \in L^2((0,T),\dot{H}^{-s+1}(\mathbb{S}))$, and integrating the equation in time, we get 
\begin{equation*}
\int_0^{T}\|\partial_tg^N\|^2_{\dot{H}^{s-1}}\leqslant (C_4+M_0)\|g_0\|^2_{\dot{H}^s}e^{(1+M_0) C_3 T}
\end{equation*}
Then we can take $M_1^2=K^2e^{(1+M_0) C_3 T}\max(\sigma^{-1},C_4+M_0)$, and we get that $g^N$ is bounded by $M_1$ in~$L^2((0,T),\dot H^{s+1}(\mathbb{S})) \cap H^1((0,T),\dot H^{s-1}(\mathbb{S}))$.

Now, we just need estimates for $\frac{\mathrm d}{\mathrm dt} J[g^N]$. We can take $h=\omega \cdot V$ in the ODE~\eqref{Galerkin} and use the tools given in the beginning of this section. We get 
\begin{align*}
\left|\frac{\mathrm d}{\mathrm dt} J[g^N]\right| &= \left| \frac{n-1}n(1-\sigma n)J[g^N] - \int_\mathbb{S}(\mathrm{Id} - \omega \otimes \omega) J[g^N] g^N \,\mathrm d\omega,\right|\\
&\leqslant (C_5+M_0C_6)\|g_0\|_{\dot{H}^s}e^{\frac12(1+M_0) C_3 T}.
\end{align*}
Indeed, again, since any component of $\mathrm{Id} - \omega \otimes \omega $ is in $\dot H^{-s}$, we can control the term~$\int_\mathbb{S}(\mathrm{Id} - \omega \otimes \omega) g^N \,\mathrm d\omega $ by any $\dot{H}^{s}$ norm of $g^N$, uniformly in $N$ and in $t\in [0,T]$.

In summary if we suppose that $g_0$ is in $\dot H^{s}(\mathbb{S})$, for some $s \in \mathbb{R}$, we have that $g^N$ is bounded in~$L^2((0,T),\dot H^{s+1}(\mathbb{S})) \cap H^1((0,T),\dot H^{s-1}(\mathbb{S}))$, and that $ J[g^N]$ and $\frac{\mathrm d}{\mathrm dt} J[g^N]$ are uniformly bounded in $N$ and $t \in [0,T]$.

Then, using weak compactness and the Ascoli-Arzela theorem, we can find an increasing sequence $N_k$, a function~$g \in L^2((0,T),\dot H^{s+1}(\mathbb{S})) \cap H^1((0,T),\dot H^{s-1}(\mathbb{S}))$, and a continuous function $J:[0,T]\rightarrow \mathbb{R}^n$ such that, as $k\rightarrow \infty $,
\begin{itemize}
\item $J[g^{N_k}]$ converges uniformly to $J$ on~$[0,T]$,
\item $g^{N_k}$ converges weakly to $g$ in $L^2((0,T),\dot H^{s+1}(\mathbb{S}))$ and in $H^1((0,T),\dot H^{s-1}(\mathbb{S}))$.
\end{itemize}

The limit $g$ is also bounded by $M_1$ in~$L^2((0,T),\dot H^{s+1}(\mathbb{S})) \cap H^1((0,T),\dot H^{s-1}(\mathbb{S}))$.
 
Then, since we have $\int_0^{T}\int_\mathbb{S}\varphi (t)\omega (g^{N_k}-g)\,\mathrm d\omega \,\mathrm dt\rightarrow 0$ for any smooth function $\varphi $, we get $\int_0^{T}\varphi (t)(J[g]-J)\,\mathrm dt=0$ and so $J=J[g]$.

For a fixed $h \in P_M$ passing the weak limit in \eqref{Galerkin} (for $N_k\geqslant M$), we get for almost every $t \in [0,T]$ that
\begin{equation*}
 \forall h \in P_M, \langle \partial_t g,h\rangle =-\sigma \langle \nabla_\omega g,\nabla_\omega h\rangle + (n-1)J[g]\cdot J[h] + \langle g,J[g]\cdot \nabla_\omega h\rangle.
\end{equation*}

And this is valid for any $M$ (except on a countable union of subsets of $[0,T]$ of zero measure).
By density (and using the first part of Lemma \ref{main_lemma}), we have that $g$ is a weak solution of our problem.

Now for any $h\in \dot H^{-s+1}(\mathbb{S})$, we have that $\langle g^N(t)-\Pi_N(g_0),h\rangle =\int_0^t\langle \partial_tg^N,h\rangle $ is controlled by $M_1\sqrt t\|h\|_{\dot H^{-s+1}}$, uniformly in $N$.
So, passing the limit, we get that~$g(t)\rightarrow g_0$ in $\dot H^{-s+1}(\mathbb{S})$ as $t\rightarrow 0$. But since we know that $g\in C([0,T],H^s(\mathbb{S}))$, by uniqueness, we get $g(0)=g_0$. 
\end{proof}

\begin{prop} Continuity with respect to the initial condition.

\label{continuity}
Set $T=\frac1{C_1}\ln(1+\frac1{1+2C_2K})$, as in \eqref{def_small_time}.
Suppose we have two solutions $g$ and $\widetilde g$, with $\|g(0)\|_{\dot H^s}\leqslant K$ and $\|\widetilde g(0)\|_{\dot H^s}\leqslant K$.

Then there exists a constant $M_3$ such that $g-\widetilde g$ is bounded in $L^2((0,T),\dot H^{s+1}(\mathbb{S}))$ and in $H^1((0,T),\dot H^{s-1}(\mathbb{S}))$ by $M_3\|g(0)-\widetilde g(0)\|_{\dot H^s}$.

\end{prop}
This automatically gives uniqueness of a weak solution on $(0,T)$ with initial condition~$g_0$.

\begin{proof}

Putting $h=(-\Delta)^sg\in \dot H^{-s+1}$ in \eqref{weak_solution_meanzero}, we do the same estimations as in the previous proposition. 
We have the same estimate as \eqref{estimate_galerkin}-\eqref{estimate_galerkin_rough}:
\begin{align}
\frac12\frac{\mathrm d}{\mathrm dt} \|g\|^2_{\dot{H}^{s}}+\sigma \|g\|^2_{\dot{H}^{s+1}}&\leqslant C_0|J[g]|\|g\|^2_{\dot{H}^{s}}+(n-1)^s|J[g]|^2\label{estimate_solution}\\
&\leqslant C_1\|g\|^2_{\dot{H}^{s}}(1+C_2\|g\|_{\dot{H}^{s}})\label{estimate_solution_rough}.
\end{align}
So if we set $T=C_1^{-1}\ln(1+(1+2C_2K)^{-1})$, we can solve this inequality on $[0,T]$, exactly as in \eqref{sol_ineq_galerkin}. 
These solutions are then uniformly bounded in $L^2((0,T),\dot H^{s+1}(\mathbb{S}))$ and in $H^1((0,T),\dot H^{s-1}(\mathbb{S}))$ (by the constant $M_1$). 

Taking $u=g-\widetilde g$, and using \eqref{weak_solution_meanzero} gives an equation for $u$: for almost all $t \in [0,T]$, for all $h \in \dot H^{-s}(\mathbb{S})$, 
\begin{equation}
 \langle \partial_t u,h\rangle =-\sigma \langle \nabla_\omega u,\nabla_\omega h\rangle + (n-1)J[u]\cdot J[h] + \langle u,J[g]\cdot \nabla_\omega h\rangle +\langle \widetilde g,J[u]\cdot \nabla_\omega h\rangle.
\label{weak_solution_diff}
\end{equation}
Now we take $h=(-\Delta)^su$ and use the first and second parts of Lemma \ref{main_lemma} to get
\begin{align}
\frac12\frac{\mathrm d}{\mathrm dt} \|u\|^2_{\dot{H}^{s}}+\sigma \|u\|^2_{\dot{H}^{s+1}}&\leqslant (1+M_1)C_3\|u\|^2_{\dot{H}^{s}}+C_7\|u\|_{\dot{H}^{s}}\|\widetilde g\|_{\dot{H}^{s+1}}\|(-\Delta)^su\|_{\dot{H}^{-s}}\nonumber\\
&\leqslant M_2(1+\|\widetilde g\|_{\dot{H}^{s+1}})\|u\|^2_{\dot{H}^{s}}\label{estimate_diff}.
\end{align}
Grönwall’s lemma gives then the following estimate:
\begin{align*}
\|u\|^2_{\dot{H}^s} + \sigma \int_0^T\|u\|^2_{\dot{H}^{s+1}} &\leqslant \|u_0\|^2_{\dot{H}^s}\exp\left(M_2\int_0^{T}(1+\|\widetilde g\|_{\dot{H}^{s+1}})\right)\\
&\leqslant \|u_0\|^2_{\dot{H}^s}e^{M_2(T+M_1^2)}.
\end{align*}
Using \eqref{weak_solution_diff}, we get that $u$ is bounded in $L^2((0,T),\dot H^{s+1}(\mathbb{S})) \cap H^1((0,T),\dot H^{s-1}(\mathbb{S}))$ by a constant $M_3$ times $\|u(0)\|_{\dot H^{s}}$.
\end{proof}

\begin{prop}Positivity for regular solutions (maximum principle).
\label{positivity}
Suppose that $g_0$ is in $\dot H^s(\mathbb{S})$, with $s$ sufficiently large (according to the Sobolev embeddings, so~$s>\frac{n+3}2$ is enough) so that the (unique) solution belongs to $C^0([0,T],C^2(\mathbb{S}))$. 
Here~$T$ is defined as in \eqref{def_small_time}, with $K=\|g_0\|_{\dot H^{s}}$.
We go back to the original formulation $f=1+g$. Then $f$ is a classical solution of \eqref{Doi_eq}.

If $f_0$ is nonnegative, then $f$ is positive for any positive time, and more precisely we have the following estimates, for all $t\in (0,T]$ and $\omega \in \mathbb{S}$ (if $f_0$ is not equal to the constant function $1$):
\begin{equation}
e^{-(n-1)\int_0^t|J[f]|}\min_\mathbb{S} f_0 <f(t,\omega)<e^{(n-1)\int_0^t|J[f]|}\max_\mathbb{S} f_0.
\label{max_principle}
\end{equation}
\end{prop}

\begin{proof}

Since the solution is in $C^0([0,T],C^2(\mathbb{S}))$, we can do the reverse integration by parts in the weak formulation \eqref{weak_solution}.
We get that, as an element of $L^2((0,T),H^{s-1}(\mathbb{S}))$, the function $\partial_tf$ is equal (almost everywhere) to $\sigma \Delta_\omega f-\nabla_\omega \cdot ((\mathrm{Id}-\omega \otimes \omega)J[f]f)$, which is an element of $C^0([0,T]\times \mathbb{S})$. 
So up to redefining it on a set of measure zero, the function $f$ belongs to $C^1([0,T],C(\mathbb{S}))\cap C^0([0,T],C^2(\mathbb{S}))$, and satisfies the partial differential equation.

Applying the chain rule and using the tools given in the beginning of this section, we get another formulation of the PDE \eqref{Doi_eq}:
\begin{equation}
\partial_tf=\sigma \Delta_\omega f-J[f]\cdot \nabla_\omega f + (n-1)J[f]\cdot \omega \, f.
\label{Doi_eq_nc}
\end{equation}

The next part of the proposition is just a classical strong maximum principle. We only prove here the left part of the inequality, the other part is very similar, once we have that $f$ is positive.

Suppose first that $f_0$ is positive.
We denote by $\widetilde T>0$ the first time such that the minimum on the unit sphere of $f$ is zero (or $\widetilde T=T$ if $f$ is always positive).

Then we have for $t\in [0,\widetilde T]$, that $\partial_tf\geqslant \sigma \Delta_\omega f-J[f]\cdot \nabla_\omega f - (n-1)|J[f]| f$. If we write $\widetilde f=f\,e^{-(n-1)\int_0^t|J[f]|}$, we get 
\begin{equation}
\partial_t\widetilde f\geqslant \sigma \Delta_\omega \widetilde f-J[f]\cdot \nabla_\omega \widetilde f.
\label{supersolution}
\end{equation}
Then the weak maximum principle (see \cite{evans1998partial}, Thm 8, §7.1.4, which is also valid on the sphere) gives us that the minimum of $\widetilde f$ on $[0,\widetilde T]\times \mathbb{S}$ is reached on $\{0\}\times \mathbb{S}$. 
That means that we have a non-strict version of the left part of the inequality \eqref{max_principle}:
\begin{equation}
\forall t\in [0,\widetilde T], \forall \omega \in \mathbb{S}, f(\omega,t)\geqslant e^{-(n-1)\int_0^t|J[f]|}\min_\mathbb{S} f_0.
\label{max_principle_weak}
\end{equation}

Consequently, we have that $\min_\mathbb{S} f(\widetilde T)>0$ and so $\widetilde T=T$. 
If now $f_0$ is only nonnegative, take $f_0^\varepsilon =\frac{f+\varepsilon}{1+\varepsilon}$, and by continuity with respect to initial condition, inequality \eqref{max_principle_weak} is still valid. 
That gives that $f$ is nonnegative on $[0, T]$, and consequently we have that inequality \eqref{supersolution} is valid on $[0,\widetilde T]$.

Now we can use the strong maximum principle (see \cite{evans1998partial}, Thm 11, §7.1.4), which gives that if the inequality \eqref{max_principle_weak} is an equality for some $t>0$ and $\omega \in \mathbb{S}$, then $\widetilde f$ is constant on $[0,t]\times \mathbb{S}$. 
So $f_0$ is the constant function $1$.
\end{proof}

\begin{prop} Global existence, positivity.
Suppose $f_0$ is a probability measure belonging to $H^s(\mathbb{S})$ (this is always the case for $s<-\frac{n-1}2$, according to Sobolev embeddings). 
Then there exists a global weak solution of \eqref{Doi_eq}, which remains a probability measure for any time.
\end{prop}
We remark that the uniqueness of the solution on any time interval remains by Proposition \ref{continuity}.

\begin{proof}
We first prove this proposition in the case $s>\frac{n+3}2$.

We define a solution by constructing it on a sequence of intervals.

We set $T_1=\frac1{C_1}\ln(1+\frac1{1+2C_2\|g_0\|_{\dot H^s}})$, as in \eqref{def_small_time}. This gives existence to a solution~$g$ in~$C([0,T_1],\dot H^s(\mathbb{S}))$. 
By induction we define $T_{k+1}=T_{k}+\frac1{C_1}\ln(1+\frac1{1+2C_2\|g(T_{k})\|_{\dot H^s}})$, which gives existence to a solution $g\in C([T_k,T_{k+1}],\dot H^s(\mathbb{S}))$.

So we have a solution on $[0,T]$, provided that $T\leqslant T_k$ for some integer $k$.

Now by the previous proposition, this solution $f=1+g$ is nonnegative. 
We obviously have $|J[g]|=|J[f]|\leqslant \int_\mathbb{S} |\omega | f=1$. 
Then we can do better estimates, starting from \eqref{estimate_solution}:
\begin{align}
\frac12\frac{\mathrm d}{\mathrm dt} \|g\|^2_{\dot{H}^{s}}+\sigma \|g\|^2_{\dot{H}^{s+1}}&\leqslant C_0|J[g]|\|g\|^2_{\dot{H}^{s}}+(n-1)^s|J[g]|^2\nonumber\\
&\leqslant C_8\|g\|^2_{\dot{H}^{s}}\label{estimate_solution_positive}.
\end{align}
Then, Grönwall’s lemma gives us that $\|g(T_{k})\|_{\dot H^s}\leqslant \|g_0\|_{\dot H^s}e^{C_8T_k}$.
Suppose now that the sequence $(T_k)$ is bounded, then $\|g(T_{k})\|_{\dot H^s}$ is also bounded. 
By the definition of~$T_{k+1}$, the difference $T_{k+1}-T_k$ does not tend to zero, which implies that the increasing sequence $(T_k)$ is unbounded, and this is a contradiction. 
So we have that~$T_k\overset{k\rightarrow \infty}{\rightarrow}\infty $, and the solution is global in time.

Now we do the general case for any $s$. 
Take $g_0^k$ a sequence of elements of $\dot H^{\frac{n}2+2}$ converging to $g_0$ in $\dot H^s$, and such that $f_0^k=1+g_0^k$ are positive functions. 
Let $g^k$ be the solutions associated to these initial conditions.

Then we have the same estimates as before, since we still have $|J[g]|\leqslant 1$, solving~\eqref{estimate_solution_positive} gives 
\begin{equation*}
\|g^k(t)\|^2_{\dot H^s}+\sigma \int_0^t\|g^k(t)\|^2_{\dot H^{s+1}}\leqslant \|g^k_0\|_{\dot H^s}e^{C_8t}.
\end{equation*}
So we can now study the difference $u=g^k-g^j$, as in \eqref{weak_solution_diff},\eqref{estimate_diff}, which satisfies, for any $h\in \dot H^{-s}(\mathbb{S})$,
\begin{equation}
 \langle \partial_t u,h\rangle =-\sigma \langle \nabla_\omega u,\nabla_\omega h\rangle + (n-1)J[u]\cdot J[h] + \langle u,J[g^k]\cdot \nabla_\omega h\rangle +\langle g^j,J[u]\cdot \nabla_\omega h\rangle.
\label{weak_solution_diff_positive}
\end{equation}
We take $h=(-\Delta)^su$ and use the first and second part of Lemma \ref{main_lemma} to get
\begin{align}
\frac12\frac{\mathrm d}{\mathrm dt} \|u\|^2_{\dot{H}^{s}}+\sigma \|u\|^2_{\dot{H}^{s+1}}&\leqslant C_9\|u\|^2_{\dot{H}^{s}}+C_7\|u\|_{\dot{H}^{s}}\|g^k\|_{\dot{H}^{s+1}}\|(-\Delta)^su\|_{\dot{H}^{-s}}\nonumber\\
&\leqslant C_{10}(1+\|\widetilde g\|_{\dot{H}^{s+1}})\|u\|^2_{\dot{H}^{s}}\label{estimate_diff_positive}.
\end{align}
If we fix $T>0$, Grönwall’s lemma gives then the following estimate:
\begin{align*}
\|u\|^2_{\dot{H}^s} + \sigma \int_0^T\|u\|^2_{\dot{H}^{s+1}} &\leqslant \|u_0\|^2_{\dot{H}^s}\exp\left(C_{10}\int_0^{T}(1+\|g^j\|_{\dot{H}^{s+1}})\right)\\
&\leqslant \|u_0\|^2_{\dot{H}^s}\exp\left(C_{10}(T+\sigma^{-1}\sqrt T\|g^k_0\|_{\dot H^s}e^{C_8T})\right).
\end{align*}
Since $\|g^k_0\|_{\dot H^s}$ is bounded (because $g^k_0$ converges in $\dot H^s$), together with \eqref{weak_solution_diff}, we finally get that $u$ is bounded in $L^2((0,T),\dot H^{s+1}(\mathbb{S})) \cap H^1((0,T),\dot H^{s-1}(\mathbb{S}))$ by a constant~$C_T$ times $\|u(0)\|_{\dot H^{s}}$. 
This gives that $g^k$ is a Cauchy sequence in that space, and then it converges to a function $g$, which is a weak solution of our problem (by Proposition~\ref{embedding_IC}, we have that $g(0)=g_0$). 
This is valid for any $T>0$, so this solution is global.

If we take $\varphi $ in $C^\infty (\mathbb{S})$, since $f^k(t)=1+g^k(t)$ is a positive function with mean~$1$, 
we have that 
\begin{equation*}
-\|\varphi \|_\infty =\langle f^k(t),-\|\varphi \|_\infty \rangle \leqslant \langle f^k(t),\varphi \rangle \leqslant \langle f^k(t),\|\varphi \|_\infty \rangle =\|\varphi \|_\infty.
\end{equation*} 
Then passing the limit gives $|\langle g(t),\varphi \rangle |\leqslant \|\varphi \|_\infty $. 
Furthermore we have~$\langle f^k(t),1\rangle =1$ so~$\langle f(t),1\rangle =1$, and if $\varphi $ is a nonnegative function, then~$\langle f^k(t),\varphi \rangle \geqslant 0$ and we get~$\langle f(t),\varphi \rangle \geqslant 0$. 
This gives that $f(t)$ is a positive radon measure with mass $1$, which is a probability measure.
\end{proof}
 
\begin{prop}
\label{regularity_boundedness}
Instantaneous regularity and boundedness estimates.
If $f_0$ is a probability measure, then the solution $f$ belongs to~$C^\infty ((0,+\infty)\times \mathbb{S})$, is positive for any time $t>0$, and we have the following estimates, for all $s\in \mathbb{R}$ and~$m\geqslant 0$:
\begin{equation*}
\|f(t)\|^2_{H^{s+m}}\leqslant C\left(1+\frac1{t^m}\right)\|f_0\|^2_{H^{s}},
\end{equation*}
where the constant $C$ depends only on $\sigma $, $s$, and $m$.

In particular we have that for $t_0>0$, f is uniformly bounded on $[t_0,+\infty)$ in any~$H^{s}$ norm.
\end{prop}

\begin{proof}
Suppose $f_0\in H^s(\mathbb{S})$, and fix $t>0$. 
The solution $f$ is in $C([0,+\infty),H^s(\mathbb{S}))$, and in $L^2((0,t),H^{s+1}(\mathbb{S}))$. 
Then there exists $s<t$ such that $f(s)\in H^{s+1}(\mathbb{S})$. 
So we can construct a solution belonging to $C([s,+\infty),H^{s+1}(\mathbb{S}))$. 
But this solution is also a weak solution in $L^2((s,T),H^{s+1}(\mathbb{S}))\cap H^{-1}((s,T),H^{s-1}(\mathbb{S}))$, for all $T>s$ so by uniqueness it is equal to $f$. 
Then~$f$ belongs to $C([t,+\infty),H^{s+1}(\mathbb{S}))$. 
Since this is true for all $t>0$, then~$f$ belongs to $C((0,+\infty),H^{s+1}(\mathbb{S}))$. 
We can repeat this argument and have that $f$ belongs to $C((0,+\infty),H^p(\mathbb{S}))$ for any $p$, and is a positive classical solution, by Proposition~\ref{positivity}. 
Using the equation, differentiating in time gives that it is also in $C^k((0,+\infty),H^p(\mathbb{S}))$ for any~$p$ and any $k$, so, by Sobolev embeddings, it is a $C^\infty $ function of $(0,+\infty)\times \mathbb{S}$.

Since we have positivity, we can have estimates for any of the modes of $f=1+g$. 
Let us denote $f^N$ the orthogonal projection of $f$ on the $N$ first eigenspaces of the Laplacian, and $g^N=f-f^N$ the projection on the other ones (high modes).

We have a Poincaré inequality on this space: $\|g^N\|^2_{\dot H^{s}}\leqslant \frac1{(N+1)(N+n-1)}\|g^N\|^2_{\dot H^{s+1}}$ (we recall that the eigenvalues of~$-\Delta $ are given by $\ell (\ell +n-2)$ for $\ell \in \mathbb{N}$). 
We use the estimate~\eqref{estimate_solution}: 
\begin{align}
\frac12\frac{\mathrm d}{\mathrm dt} \|g\|^2_{\dot{H}^{s}}+\sigma \|g&\|^2_{\dot{H}^{s+1}}\leqslant C_0|J[g]|\|g\|^2_{\dot{H}^{s}}+(n-1)^s|J[g]|^2\nonumber\\
&\leqslant \tfrac{C_0}{(N+1)(N+n-1)}\|g\|^2_{\dot{H}^{s+1}}+(n-1)^s|J[g]|^2+C_0\|f^N-1\|^2_{\dot{H}^{s}}.
\label{high-low}
\end{align}
Now we have, since $f$ is a probability measure, that 
\begin{equation*}
\|f^N-1\|^2_{\dot{H}^{s}}=\int_{\mathbb{S}}(-\Delta)^sf^N f\mathrm d\omega \leqslant \|(-\Delta)^sf^N\|_{L^\infty}\leqslant K_N\|f^N-1\|_{\dot{H}^{s}},
\end{equation*}
the last inequality being the equivalence between norms in finite dimension.
Dividing by this last norm, this gives that the low modes of $f$ are uniformly bounded in time by a constant $K_N$. Then we have, taking $N$ sufficiently large,
\begin{equation*}
\frac12\frac{\mathrm d}{\mathrm dt} \|g\|^2_{\dot{H}^{s}}+\frac{\sigma}2\|g\|^2_{\dot{H}^{s+1}}\leqslant C_{11},
\end{equation*}
Now multiplying by $t$ this formula at order $s+1$, we get
\begin{equation*}
\frac12\frac{\mathrm d}{\mathrm dt}(t \|g\|^2_{\dot{H}^{s+1}})+\frac{\sigma}2 t \|g\|^2_{\dot{H}^{s+2}}\leqslant C_{12}t + \frac12 \|g\|^2_{\dot{H}^{s+1}},
\end{equation*}
and finally
\begin{equation*}
\frac12\frac{\mathrm d}{\mathrm dt}( \|g\|^2_{\dot{H}^{s}}+\tfrac{\sigma}2 t \|g\|^2_{\dot{H}^{s+1}})+ \tfrac{\sigma}4(\|g\|^2_{\dot{H}^{s+1}}+\tfrac{\sigma}2 t \|g\|^2_{\dot{H}^{s+2}})\leqslant C_{11}+C_{12}\tfrac{\sigma}2t.
\end{equation*}
Together with Poincaré inequality, solving this inequality gives us 
\begin{equation*}
\|g\|^2_{\dot{H}^{s}}+\frac{\sigma}2 t \|g\|^2_{\dot{H}^{s+1}}\leqslant \|g_0\|^2_{\dot{H}^{s}}e^{-(n-1)\frac{\sigma}4t}+C_{13}(1+t).
\end{equation*}
So we have the result for $\|f\|^2_{{H}^{s}}=1+\|g\|^2_{\dot{H}^{s}}$, and $m=1$: 
\begin{equation*}
\|f(t)\|^2_{H^{s+1}}\leqslant C\left(1+\frac1{t}\right)\|f_0\|^2_{H^{s}}.
\end{equation*}
Then we apply this inequality between $0$ and $\frac{t}2$, and the inequality at order $m$ between~$\frac{t}2$ and $t$ to get the result at order $m+1$. 
The case where $m$ is any nonnegative real also works, by interpolation. 
\end{proof}

This last proposition ends the proof of Theorem \ref{generalresults}. Let us do here two small comments concerning the analyticity of the solution and the limit case with no noise: $\sigma =0$.
\begin{remark}
Analyticity of the solution. 
We can show, as claimed in \cite{constantin2004remarks}, \cite{constantin2005dissipativity} that at any time $t>0$ the solution is analytic in the space variable. 
The idea is to show, following~\cite{constantin2005dissipativity} (based on \cite{cao2000gevrey}, \cite{foias1989gevrey}), that the solution is in some Gevrey class of functions, defined by a parameter depending on time. This class is a subset of the set of real analytic functions on the sphere. 
More details and a complete proof are given in Appendix \ref{analyticity}.
We could have directly dealt with this classes of functions instead of working in the Sobolev spaces, but we will not need these properties of analyticity in the following. 
In any case, to prove analyticity we need the initial condition to be in $H^{-\frac{n-1}2}(\mathbb{S})$, so this study of instantaneous regularization was necessary.
\end{remark}

\begin{remark}\label{no_noise}
Case where $\sigma =0$: no noise.
The proof is also valid, except that the solution belongs to $L^\infty ((0,T),H^{s}(\mathbb{S}))\cap H^1((0,T),H^{s-1}(\mathbb{S}))$ if the initial condition is in $H^s(\mathbb{S})$. 
By an optimal regularity argument, we can get that a solution is in fact in~$C([0,T],H^{s}(\mathbb{S}))$.
The nonnegativity argument is then also valid, and so the solution is global.
Obviously, we do not have the instantaneous regularity and boundedness estimates.
\end{remark}

\section{Using the free energy}
\label{using_free_energy}
In this section, we derive the Onsager free energy \eqref{Onsager_free_energy} for Doi equation \eqref{Doi_eq}, and use it to get general results on the steady states.
\subsection{Free energy and steady states}
We rewrite the equation \eqref{Doi_eq}:
\begin{equation*}
\partial_t f = Q(f) = \nabla_\omega \cdot (\sigma \nabla_\omega f - \nabla_\omega (\omega \cdot J[f])f) = \nabla_\omega \cdot (f \nabla_\omega (\sigma \ln f - \omega \cdot J[f])).
\end{equation*}
Since any solution is in $C^\infty ((0,+\infty)\times \mathbb{S})$, and positive for any $t>0$, there is no problem with using $\ln f$, and doing any integration by parts. 
We multiply the equation by~$\sigma \ln f - \omega \cdot J[f]$ and integrate by parts, we get
\begin{equation*}
\int_\mathbb{S} \partial_t f(\sigma \ln f - \omega \cdot J[f])\, \mathrm d\omega = - \int_\mathbb{S} f |\nabla_\omega (\sigma \ln f - \omega \cdot J[f])|^2\, \mathrm d\omega.
\end{equation*}
Since the left part can be recast as a time derivative, this is a conservation relation. 
We define the free energy $\mathcal F(f)$ and the dissipation term $\mathcal D(f)$ by

\begin{gather}
\label{def_F}
\mathcal F(f)=\sigma \int_\mathbb{S} f \ln f - \tfrac12|J[f]|^2,\\
\label{def_D}
\mathcal D(f)=\int_\mathbb{S} f |\nabla_\omega (\sigma \ln f - \omega \cdot J[f])|^2,
\end{gather}
and we have the following conservation relation:
\begin{equation}
\frac{\mathrm d}{\mathrm dt}\mathcal F+\mathcal D=0
\label{free_energy}
\end{equation}

We define a steady state as a (weak) solution which does not depend on time. 
Here are some characterizations of the steady states.

\begin{prop}Steady states.
The steady states of Doi equation \eqref{Doi_eq} are the probability measures $f$ on $\mathbb{S}$ which satisfy one of the following equivalent conditions.
\begin{enumerate}
\item Equilibrium: $f\in C^2(\mathbb{S})$ and $Q(f)=0$
\item No dissipation: $f\in C^1(\mathbb{S})$ and $\mathcal D(f)=0$
\item The probability density $f\in C^0(\mathbb{S})$ is positive and a critical point of $\mathcal F$ (under the constraint of mean $1$).
\item There exists $C\in \mathbb{R}$ such that $\sigma \ln f-J[f]\cdot \omega =C$.
\end{enumerate}
\end{prop}

\begin{proof}
By definition, a steady state $f$ is a solution independent of $t$. 
Since it is a solution, it is positive and $C^\infty $, and we get that $Q(f)=0$.
By the conservation relation \eqref{free_energy}, we get that $\frac{\mathrm d}{\mathrm dt}\mathcal F=0$, so $\mathcal D(f)=0$. 
Since it is positive, we get that~$\nabla_\omega (\sigma \ln f - \omega \cdot J[f])=0$, so there exists $C\in \mathbb{R}$ such that $\sigma \ln f-J[f]\cdot \omega =C$.

Now we do a variational study of $\mathcal F$ around $f$. 
We take a small perturbation~$f+h$ of~$f$ which remains a probability density function (which means that $\int_\mathbb{S} h=0$).

We can expand the function $x\mapsto x\ln x$ around $f$, since $f\geqslant \varepsilon >0$, and we have
\begin{align*}
\mathcal F(f+h)&=\sigma \int_\mathbb{S}(f\ln f+h\ln f + h) - \frac12|J[f]|^2 - J[f]\cdot \int_\mathbb{S} \omega h +O(\|h\|_\infty^2)\\
&=\mathcal F(f)+\int_\mathbb{S} h(\sigma \ln f - J[f]\cdot \omega)+O(\|h\|_\infty^2),\\
&=\mathcal F(f)+O(\|h\|_\infty^2),
\end{align*}
which means that $f$ is a critical point of $\mathcal F$. So $f$ satisfies the four conditions.

Conversely if $f\in C^2(\mathbb{S})$ and $Q(f)=0$, then $f$ is obviously a steady-state. 

If $\sigma \ln f-J[f]\cdot \omega =C$, then $f\in C^2(\mathbb{S})$ and $Q(f)=0$.
We will show that the second and third conditions reduce to this fourth condition.

Doing the above computation around a positive $f\in C^0(\mathbb{S})$ gives that if $f$ is a critical point for the free energy, then~$\int_\mathbb{S} h(\sigma \ln f - J[f]\cdot \omega)$ is zero for any $h$ with mean zero.
This is exactly saying that~$\sigma \ln f - J[f]\cdot \omega $ is constant.

Finally if we suppose $f\in C^1(\mathbb{S})$ and $\mathcal D(f)=0$, at any point $\omega_0\in \mathbb{S}$ such that~$f(\omega_0)>0$ we have that $\nabla (\sigma \ln f-J[f]\cdot \omega)=0$ on a neighborhood of $\omega_0$. 

The function $\varphi $ defined by $\varphi (\omega)=\sigma \ln f-J[f]\cdot \omega $ is then locally constant at any point where it is finite, so~$\varphi^{-1}(\{C\})$ is open in $\mathbb{S}$ for any $C\in \mathbb{R}$. 

Now if $\varphi (\omega_k)=C$, with $\omega_k$ converging to $\omega_\infty $, then $f(\omega_k)=\exp(\frac{C+J[f]\cdot \omega_k}{\sigma})$. 
Passing to the limit, we get that $f(\omega_\infty)=\exp(\frac{C+J[f]\cdot \omega_\infty}{\sigma})$, which gives~$\varphi (\omega_\infty)=C$. 
So $\varphi^{-1}(\{C\})$ is closed. 

Since $f$ is not identically zero, there exists $C\in \mathbb{R}$ such \mbox{that $\varphi^{-1}(\{C\})\neq \emptyset$}, and by connectedness of the sphere, we get $\varphi^{-1}(\{C\})=\mathbb{S}$, so~$\sigma \ln f - J[f]\cdot \omega =C$.
\end{proof}

\subsection{LaSalle principle}
We give here an adaptation of LaSalle’s invariance principle to our PDE framework.
\begin{prop}
\label{lasalle}LaSalle’s invariance principle.
Let $f_0$ be a probability measure on the sphere $\mathbb{S}$. 
We denote by $\mathcal F_\infty $ the limit of~$\mathcal F(f(t))$ as $t\rightarrow \infty $, where $f$ is the solution to Doi equation \eqref{Doi_eq} with initial condition $f_0$.

Then the set $\mathcal E_\infty =\{f \in C^\infty (\mathbb{S}) \text{ s.t. } \mathcal D(f)=0 \text{ and } \mathcal F(f)=\mathcal F_\infty \}$ is not empty.

Furthermore $f(t)$ converges in any $H^s$ norm to this set of equilibria (in the following sense):
\begin{equation*}
\lim_{t\rightarrow \infty}\, \inf_{g\in \mathcal E_\infty}\|f(t)-g\|_{H^s}=0.
\end{equation*}
\end{prop}

\begin{proof}

First of all $\mathcal F(f(t))$ is decreasing in time, and bounded below by $-\frac12$, so $\mathcal F_\infty $ is well defined.

Let $(t_n)$ be an unbounded increasing sequence, and suppose that $f(t_n)$ converges in~$H^s(\mathbb{S})$ to $f_\infty $ for some $s\in \mathbb{R}$.
We first remark that $f(t_n)$ is uniformly bounded in~$H^{s+2p}(\mathbb{S})$ (using Theorem \ref{generalresults}), and then by a simple interpolation estimate we get that~$\|f(t_n)-f(t_m)\|^2_{\dot H^{s+p}}\leqslant \|f(t_n)-f(t_m)\|_{\dot H^{s}}\|f(t_n)-f(t_m)\|_{\dot H^{s+2p}}$, and $f(t_n)$ also converges in $H^{s+p}(\mathbb{S})$. 
So $f_\infty $ is in any $H^s(\mathbb{S})$. 

We want to prove that $\mathcal D(f_\infty)=0$. 
Supposing this is not the case, we write
\begin{align}
\mathcal D(f)&=\sigma^2\int_\mathbb{S}\frac{|\nabla_\omega f|^2}{f} + J[f]\cdot \int_\mathbb{S}(\mathrm{Id} - \omega \otimes \omega)f\, J[f] - 2 \sigma J[f]\cdot \int_\mathbb{S}\nabla_\omega f\nonumber\\
&=\sigma^2\int_\mathbb{S}\frac{|\nabla_\omega f|^2}{f} +(1-2(n-1)\sigma)|J[f]|^2 - \int_\mathbb{S} (\omega \cdot J[f])^2f.
\label{expression_D}
\end{align}
Now we take $s$ sufficiently large such that $H^{s}(\mathbb{S})\subset L_\infty (\mathbb{S})\cap H^1(\mathbb{S})$. 
If $f_\infty $ is positive, then~$\mathcal D$, as a function from the nonnegative elements of $H^{s}(\mathbb{S})$ to $[0,+\infty ]$, is continuous at the point $f_\infty $. 
In particular since $\mathcal D(f_\infty)>0$, there exist $\delta >0$ and~$M>0$ such that if $\|f-f_\infty \|_{H^s}\leqslant \delta $, then we have~$\mathcal D(f)\geqslant M$.
We want to show the same result in the case where $f_\infty $ is only nonnegative.
We define 
\begin{equation*}
\mathcal D_\varepsilon (f)=\sigma^2\int_\mathbb{S}\frac{|\nabla_\omega f|^2}{f+\varepsilon} +(1-2(n-1)\sigma)|J[f]|^2 - \int_\mathbb{S} (\omega \cdot J[f])^2f.
\end{equation*}
We have that by monotone convergence that $\mathcal D_\varepsilon (f_\infty)$ converges to $\mathcal D(f_\infty)$ as~$\varepsilon \rightarrow 0$. 
So there exists $\varepsilon >0$ such that $\mathcal D_\varepsilon (f_\infty)>0$. 
Now by continuity of $\mathcal D_\varepsilon $ at the point~$f_\infty $, we get that there exists $\delta >0$ and $M>0$ such that if $\|f-f_\infty \|_{H^s}\leqslant \delta $, then $\mathcal D_\varepsilon (f)\geqslant M$. 
And the fact that $\mathcal D(f) \geqslant \mathcal D_\varepsilon (f)$ gives the same result as before.

Now since $\partial_tf$ is uniformly bounded in $H^s$ (for $t\geqslant t_1>0$), there exists~$\tau >0$ such that if $|t-t'|\leqslant \tau $, then $\|f(t)-f(t')\|_{H^s}\leqslant \frac{\delta}2$. 
We take then $N$ sufficiently large such that~$\|f(t_n)-f_\infty \|_{H^s}\leqslant \frac{\delta}2$ for all $n\geqslant N$.

Then we have that for $n\geqslant N$, $\mathcal D(f)\geqslant M$ on $[t_n,t_n+\tau ]$. Up to extracting, we can assume that $t_{n+1}\geqslant t_n+\tau $, so we have
\begin{equation*}
\mathcal F(f(t_N))-\mathcal F(f(t_{N+p}))=\int^{t_{N+p}}_{t_N}\mathcal D(f)\geqslant p\tau M.
\end{equation*}
Since the left term is bounded by $\mathcal F(f(t_N))-\mathcal F_\infty $, taking $p$ sufficiently large gives the contradiction.

Now if we suppose that for a given $s$ the distance (in $H^s$ norm) between $f(t)$ and~$\mathcal E_\infty $ does not tend to $0$, we get $\varepsilon >0$ and a sequence $t_n$ such that for all $g\in \mathcal E_\infty $, we have $\|f(t_n)-g\|_{H^s}\geqslant \varepsilon $.
Since $f(t_n)$ is bounded in $H^{s+1}(\mathbb{S})$, by a compact Sobolev embedding, up to extracting we can assume that $f(t_n)$ is converging in~$H^{s}(\mathbb{S})$ to~$f_\infty $. 
By the previous argument $f\in C^\infty (\mathbb{S})$ and we have $\mathcal D(f_\infty)=0$. 
Obviously since~$\mathcal F(f)$ is decreasing in time we have that $\mathcal F(f_\infty)=\mathcal F_\infty $. So $f_\infty $ belongs to $\mathcal E_\infty $, and then~$\|f(t_n)-f_\infty \|_{H^s}\geqslant \varepsilon $ for all $n$. 
This is a contradiction.

Since the distance between $f(t)$ and $\mathcal E_\infty $ tends to $0$, obviously this set is not empty.
\end{proof}

\subsection{Computation of equilibria}

Define, for a unit vector $\Omega \in \mathbb{S}$, and $\kappa \geqslant 0$ the Fisher-Von Mises distribution with concentration parameter $\kappa $ and orientation $\Omega $ by 
\begin{equation}
 M_{\kappa \Omega}(\omega)=\dfrac{\exp(\kappa \,\omega \cdot \Omega)}{\int_{\mathbb{S}}\exp(\kappa \,\upsilon \cdot \Omega )\mathrm d\upsilon} \,.
\end{equation}
Note that the denominator depends only on $\kappa $. 
We have that the density of $M_{\kappa \Omega}$ is~$1$, and the flux is
\begin{equation}
 J[M_{\kappa \Omega}]=\dfrac{\int_{\mathbb{S}}\omega \exp(\kappa \,\omega \cdot \Omega)\mathrm d\omega}{\int_{\mathbb{S}}\exp(\kappa \,\omega \cdot \Omega)\mathrm d\omega} = c(\kappa)\Omega,
\label{flux_M}
\end{equation}
where 
\begin{equation}
 c(\kappa)=\dfrac{\int_0^\pi \cos\theta \, e^{\kappa \cos\theta}\sin^{n-2}\theta \, \mathrm d\theta}{\int_0^\pi e^{\kappa \cos\theta}\sin^{n-2}\theta \, \mathrm d\theta}.
\end{equation}
If $f$ is an equilibrium, $\sigma \ln f - J[f]\cdot \omega $ is constant, and then~$f=C\exp(\sigma^{-1}J[f]\cdot \omega)$.
Since $f$ is a probability density function, we get $f=M_{\kappa \Omega}$ with $\kappa \Omega =\sigma^{-1}J[f]$ (in the case where $|J[f]|=0$, then $\kappa =0$ and we can take any $\Omega $, this is just the uniform distribution). 
Finally with \eqref{flux_M} we get $J[f]= c(\kappa)\Omega $, which gives the following compatibility condition
\begin{equation}
c(\kappa)=\sigma \kappa.
\label{compatibility_condition}
\end{equation}
We give the solutions of this equation in a proposition.

\begin{prop} Compatibility condition
\begin{itemize}
\item If $\sigma \geqslant \frac1n$, there is only one solution to the compatibility condition: $\kappa =0$. 
The only equilibrium is the constant function $f=1$.
\item If $\sigma <\frac1n$, the compatibility condition has exactly two solutions: $\kappa =0$ and one unique positive solution, that we will denote $\kappa (\sigma)$.
Apart from the constant function $f=1$ (the case $\kappa =0$), the equilibria form a manifold of dimension~$n-1$: the functions of the form $f=M_{\kappa (\sigma)\Omega}$, where $\Omega \in \mathbb{S}$ is an arbitrary unit vector.
\end{itemize}
\end{prop}
\begin{proof}

Let us denote $\widetilde{\sigma}(\kappa)=\frac{c(\kappa)}{\kappa}$.
A simple Taylor expansion gives $\widetilde{\sigma}(\kappa)\underset{\kappa \rightarrow 0}{\rightarrow}\frac1n$. 
Since the function $\widetilde{\sigma}$ tends to $0$ as~\mbox{$\kappa \rightarrow +\infty $} (because $c(\kappa)\leqslant 1$), it is sufficient to prove that it is decreasing. 
Indeed the function is then a one-to-one correspondence from~$\mathbb{R}_+^*$ to $(0,\frac1n)$, and the compatibility condition for $\kappa >0$ is exactly solving $\sigma =\widetilde{\sigma}(\kappa)$.

But we have (after one integration by parts) that $\widetilde{\sigma}'(\kappa)=\frac1{\kappa}(1-n\widetilde{\sigma}(\kappa)-c(\kappa)^2)$, which, by the following lemma is negative for $\kappa >0$. 
\end{proof}

\begin{lemma}
Define $\beta =c(\kappa)^2+n\widetilde{\sigma}(\kappa)-1$. Then for any $\kappa >0$, we have $\beta >0$.
\label{beta_positive} 
\end{lemma}

\begin{proof}

Define $[\gamma (\cos \theta)]_\kappa =\int_0^\pi \gamma (\cos\theta) \, e^{\kappa \cos\theta}\sin^{n-2}\theta \, \mathrm d\theta $. 

Then we have by definition $\beta = \dfrac{\kappa [\cos \theta ]_\kappa^2 + n [\cos \theta ]_\kappa [1]_\kappa - \kappa [1]_\kappa^2}{\kappa [1]_\kappa^2}$. 
So we only have to show that the numerator is positive. 
We will prove in fact that the Taylor expansion of this term in $\kappa $ has only positive terms.

We have, if we denote $a_p=\frac1{(2p)!}\int_0^\pi \cos^{2p}\theta \sin^{n-2}\theta \, \mathrm d\theta \geqslant 0$,
\begin{equation*}
[1]_\kappa =\sum_{p=0}^\infty a_p \kappa^{2p}\quad\text{ and }\quad[\cos\theta ]_\kappa =\sum_{p=0}^\infty (2p+2)\, a_{p+1}\kappa^{2p+1}
\end{equation*}
Now doing an integration by part in the definition of $a_{p+1}$, we get 
\begin{equation*}
a_{p+1}=\dfrac{2p+1}{n-1} \left(\dfrac{a_p}{(2p+1)(2p+2)} - a_{p+1}\right),
\end{equation*}
which gives 
\begin{equation}
\label{induction_ap}
(2p+2)\, a_{p+1}=\dfrac{a_p}{2p+n}.
\end{equation}
We have, for $\kappa >0$,
\begin{align*}
 \beta \kappa [1]_\kappa^2 &= \sum_{k=0}^\infty \left(\sum_{p+q=k-1\hspace{-0,7cm}}(2p+2)\, a_{p+1} (2q+2)\, a_{q+1} + \sum_{p+q=k\hspace{-0,4cm}}n(2p+2)\, a_{p+1} a_{q} - a_p a_q\right) \kappa^{2k+1}\\
&= \sum_{k=0}^\infty \left(\sum_{p+q=k, p\geqslant 1}2p\, a_{p} \tfrac{1}{2q+n} a_q + \sum_{p+q=k}(\tfrac{n}{2p+n}-1)\, a_{p} a_{q} \right) \kappa^{2k+1}\\
&= \sum_{k=0}^\infty \left(\sum_{p+q=k}2p\left(\tfrac{1}{2q+n} - \tfrac{1}{2p+n}\right)\, a_{p} a_{q} \right) \kappa^{2k+1}\\
&=\sum_{k=0}^\infty \left(\sum_{p+q=k}\left(p(\tfrac{1}{2q+n} - \tfrac{1}{2p+n}) +q(\tfrac{1}{2p+n} - \tfrac{1}{2q+n})\right) \, a_{p} a_{q} \right) \kappa^{2k+1}\\
&= \sum_{k=0}^\infty \left(\sum_{p+q=k}\tfrac{2(p-q)^2}{(2p+n)(2q+n)}\, a_{p} a_{q} \right) \kappa^{2k+1}
\end{align*}
So we finally get 
\begin{equation*}
 \beta =\left(\sum_{p=0}^\infty a_p \kappa^{2p}\right)^{-1}\sum_{k=0}^\infty \left(\sum_{p+q=k}\tfrac{2(p-q)^2}{(2p+n)(2q+n)}\, a_{p} a_{q} \right) \kappa^{2k},
\end{equation*}
which gives that $\beta >0$ when $\kappa >0$.
\end{proof}

\begin{remark}
We can do another proof, following an argument of \cite{zhou2005new}, which does not need to compute explicitly $\beta $. 

The idea is that we compute $\widetilde{\sigma}''=(n-1)\frac{\beta}{\kappa^2}-2\widetilde{\sigma}(\widetilde{\sigma}-\beta)$, 
so we see (except in the case $\kappa =0$) that if $\widetilde{\sigma}'=-\frac{\beta}{\kappa}=0$, then $\widetilde{\sigma}''<0$ (indeed, we will easily see in \eqref{cos2bis} that $\widetilde{\sigma}-\beta $ is positive). 
For the case $\kappa =0$, we can compute the Taylor expansion of~$\widetilde{\sigma}$ up to order $2$: 
$\widetilde{\sigma}(\kappa)=\frac1n-\frac1{n^2(n+2)}\kappa^2+O(\kappa^4)$. 
So we have that any critical point of $\widetilde{\sigma}$ is a maximum. 
Since there is a local maximum at $\kappa =0$ then the function is decreasing.
\end{remark}
We can have an asymptotic expansion of the order parameter $c(\kappa (\sigma))$ as $\sigma $ reaches the critical value $\frac1n$.
Indeed we have that $\sigma -\frac1n\sim-\frac1{n^2(n+2)}\kappa (\sigma)^2$ by the expansion of $\widetilde{\sigma}$ in the previous remark.
So 
\begin{equation}
c(\kappa (\sigma))\sim\tfrac1n \kappa (\sigma)\sim\sqrt{(n+2)(\tfrac1n-\sigma)}\text{ as }\sigma \rightarrow \tfrac1n.
\label{expansion_coefficients}
\end{equation}
\begin{prop} Minimum of the free energy
\label{minimum_free_energy}
\begin{itemize}
\item If $\sigma \geqslant \frac1n$, the minimum of the free energy is $0$, only reached by the uniform distribution. Any solution converges to the uniform distribution in any $H^s$ norm.
\item If $\sigma <\frac1n$, the minimum of the free energy is negative, only reached by any non-isotropic equilibrium $M_{\kappa (\sigma)\Omega}$.
\end{itemize}
\end{prop}
\begin{proof}
By LaSalle principle (Proposition \ref{lasalle}), we have that 
\begin{equation*}
\min_{f\in C^\infty (\mathbb{S}),\, f>0}\mathcal F(f)=\min_{f\in C^\infty (\mathbb{S}),\, f>0,\, \mathcal D(f)=0}\mathcal F(f).
\end{equation*}
Indeed for any positive initial condition $f$ in $C^\infty (\mathbb{S})$, there exists an equilibrium $f_\infty $ such that $\mathcal F(f_\infty)=\mathcal F_\infty \leqslant \mathcal F(f)$. 
This gives 
\begin{equation*}
\inf_{f\in C^\infty (\mathbb{S}),\, f>0}\mathcal F(f)=\inf_{f\in C^\infty (\mathbb{S}),\, f>0,\, \mathcal D(f)=0}\mathcal F(f).
\end{equation*}
Since the set of equilibria is compact (either a single point or one point and a manifold homeomorphic to $\mathbb{S}$), this infimum is a minimum.

Furthermore, if $f_0$ is not an equilibrium, then $\mathcal D(f_0)>0$, and then $\mathcal F(f(t))$ is decreasing in the neighborhood of $t=0$. So the minimum of $\mathcal F$ cannot be reached for $f_0$.

In the case $\sigma \geqslant \frac1n$, this gives the result since the only equilibrium is the constant function $1$. By LaSalle principle, we also get that the solution is converging to in any $H^s$ norm. 

In the case $\sigma <\frac1n$, we have that $\mathcal F(1+\varepsilon \omega \cdot \Omega)\sim\frac1n(\sigma -\frac1n)\varepsilon^2$ for a fixed unit vector~$\Omega \in \mathbb{S}$, so there exists $f_0$ such that $\mathcal F(f_0)<0$. 
Then the uniform distribution cannot be a global minimizer. 
Since $\mathcal F(M_{\kappa (\sigma)\Omega})$ is independent of $\Omega $, we get that this value is the minimum.
\end{proof}

\section{Convergence to equilibrium}
\label{convergence}
In this section, we establish and study the convergence of the solution to an equilibrium for any initial condition, in the three different regimes, depending wether $\sigma $ is greater, less, or equal to $\frac1n$.
\subsection{A new entropy, application to the subcritical case $\sigma >\frac1n$.}

In this section we derive a convex entropy, which shows global decay to the uniform distribution in the case $\sigma >\frac1n$.

We define on $\dot H^{-\frac{n-1}2}(\mathbb{S})$ the norm $\|\cdot \|_{\widetilde H^{-\frac{n-1}2}}$ by $\|g\|^2_{\widetilde H^{-\frac{n-1}2}}=\int_{\mathbb{S}}g\widetilde{\Delta}_{n-1}^{-1}g$, where the conformal Laplacian~$\widetilde{\Delta}_{n-1}$ is defined by \eqref{conformal_laplacian}. 
This norm is equivalent to $\|\cdot \|_{\dot H^{-\frac{n-1}2}}$. 
We also define $\|\cdot \|_{\widetilde H^{-\frac{n-3}2}}$ by $\|g\|^2_{\widetilde H^{-\frac{n-3}2}}=\int_{\mathbb{S}}\Delta g\widetilde{\Delta}_{n-1}^{-1}g$, and this norm is equivalent to the $\|\cdot \|_{\dot H^{-\frac{n-3}2}}$ norm.

Taking $h=\widetilde{\Delta}_{n-1}^{-1}g$ in the weak formulation \eqref{weak_solution_meanzero}, and using the last part of Lemma~\ref{main_lemma}, we obtain a conservation relation:
\begin{equation}
 \frac12\frac{\mathrm d}{\mathrm dt}\|g\|^2_{\widetilde H^{-\frac{n-1}2}}=-\sigma \|g\|^2_{\widetilde H^{-\frac{n-3}2}} + \frac1{(n-2)!}|J[g]|^2.
\label{conservation_relation}
\end{equation}
We remark that this is a conservation law between quadratic quantities, as it would be the case for a linear equation.

Since the component of $g$ on the space of spherical harmonics of degree $1$ is given by $n \omega \cdot J[g]$, a simple computation shows that the contribution to $\|g\|^2_{\widetilde H^{-\frac{n-1}2}}$ of this component is equal to $\frac{n}{(n-1)!}|J[g]|^2$. 
Then the last term of the conservation relation~\eqref{conservation_relation} is bounded by~$\frac{n-1}{n}\|g\|^2_{\widetilde H^{-\frac{n-1}2}}$. 
Together with the Poincaré inequality~$\|g\|^2_{\widetilde H^{-\frac{n-3}2}}\geqslant (n-1) \|g\|^2_{\widetilde H^{-\frac{n-3}2}}$, we get the following estimate: 
\begin{equation*}
\frac12\frac{\mathrm d}{\mathrm dt}\|g\|^2_{\widetilde H^{-\frac{n-1}2}}\leqslant (n-1)(\tfrac1n-\sigma)\|g\|^2_{\widetilde H^{-\frac{n-1}2}}.
\end{equation*}
This gives in the case $\sigma >\frac1n$ an exponential decay of rate $(n-1)(\sigma -\frac1n)$ for the norm~$\|\cdot \|_{\widetilde H^{-\frac{n-1}2}}$: 
\begin{equation*}
\|g\|_{\widetilde H^{-\frac{n-1}2}}\leqslant \|g_0\|_{\widetilde H^{-\frac{n-1}2}}\exp(-(n-1)(\sigma -\tfrac1n)t).
\end{equation*}
In the general case, if $f_0\in H^s(\mathbb{S})$ with $s>-\frac{n-1}2$, we use the estimate \eqref{high-low}: 
\begin{equation*}
\frac12\frac{\mathrm d}{\mathrm dt} \|g\|^2_{\dot{H}^{s}}+\sigma \|g\|^2_{\dot{H}^{s+1}}\leqslant \tfrac{C_0}{(N+1)(N+n-1)}\|g\|^2_{\dot{H}^{s+1}}+(n-1)^s|J[g]|^2+C_0\|f^N-1\|^2_{\dot{H}^{s}}.
\end{equation*}
Now we have, since $f$ is a probability measure,
\begin{equation*}
(n-1)^s|J[g]|^2+\|f^N-1\|^2_{\dot{H}^{s}}\leqslant K_N\|f^N-1\|^2_{\widetilde H^{-\frac{n-1}2}}\leqslant K_N\|g_0\|^2_{\widetilde H^{-\frac{n-1}2}}e^{-2(n-1)(\sigma -\frac1n)t},
\end{equation*}
the first inequality being the equivalence between norms in finite dimension. 
For any~$\varepsilon <\frac1n$, taking $N$ sufficiently large, together with Poincaré inequality we get
\begin{equation*}
\frac12\frac{\mathrm d}{\mathrm dt} \|g\|^2_{\dot{H}^{s}}+(n-1)(\sigma -\varepsilon)\|g\|^2_{\dot{H}^{s}}\leqslant C\|g_0\|^2_{\widetilde H^{-\frac{n-1}2}}e^{-2(n-1)(\sigma -\frac1n)t},
\end{equation*}
where the constant $C$ depends only on $s$.

Solving this equation, we get 
\begin{equation*}
\|g\|^2_{\dot{H}^{s}}\leqslant \|g_0\|^2_{\dot{H}^{s}}e^{-2(n-1)(\sigma -\varepsilon)t}+\tfrac{C}{(n-1)(\frac1n-\varepsilon)}\|g_0\|^2_{\widetilde H^{-\frac{n-1}2}}e^{-2(n-1)(\sigma -\frac1n)t}.
\end{equation*}
Taking for example $\varepsilon =\frac1{2n}$, since $s>-\frac{n-1}2$, we get
\begin{equation*}
\|g\|^2_{\dot{H}^{s}}\leqslant (1+2\widetilde C\tfrac{n}{n-1})\|g_0\|^2_{\dot{H}^{s}}e^{-2(n-1)(\sigma -\frac1n)t}.
\end{equation*}
In summary, we have the following theorem:

\begin{thm} New entropy.
\label{new_entropy}
For a given probability density function $f$, we define the quantities $\mathcal{H}(f)=\|f-1\|^2_{\widetilde H^{-\frac{n-1}2}}$ and $\widetilde{\mathcal D}(f)=2\sigma \|f-1\|^2_{\widetilde H^{-\frac{n-3}2}} - \frac2{(n-2)!}|J[f]|^2$. 

We have a conservation relation, for any solution $f$ of Doi equation \eqref{Doi_eq}:
\begin{equation}
 \frac{\mathrm d}{\mathrm dt}\mathcal{H}(f)+\widetilde{\mathcal D}(f)=0.
\label{conservation_free_energy}
\end{equation}
When $\sigma \geqslant \frac1n$, the term $\widetilde{\mathcal D}(f)$ is nonnegative, so the new entropy $\mathcal{H}(f)$ is decreasing in time.

Furthermore, if $\sigma >\frac1n$, then in any Sobolev space $H^s(\mathbb{S})$ with $s\geqslant -\frac{n-1}2$, we have global exponential decay of the solution to the uniform distribution, with rate given by $(n-1)(\sigma -\frac1n)$.

More precisely there is a constant $C$ depending only on $s$ such that for all initial condition $f_0\in H^s(\mathbb{S})$, we have
\begin{equation*}
\|f-1\|_{H^{s}}\leqslant C\|f_0-1\|_{H^{s}}e^{-(n-1)(\sigma -\tfrac1n)t}.
\end{equation*}
\end{thm}

Let us do a small remark here. 
Actually this conservation relation is true for any solution, without any positivity condition. 
We only need the mean of $f$ to be $1$. 
And since we have existence and uniqueness in small time for any initial condition, with the same instantaneous regularity results (only valid for a short time existence), we get that the solution belongs to $H^{-\frac{n-1}2}(\mathbb{S})$ at some time. 
But the conservation relation gives then that we have a global solution. 
So we can state a stronger theorem of existence and uniqueness: 

\begin{thm}
Given an initial condition $f_0$ in $H^s(\mathbb{S})$ (not necessarily nonnegative), there exists a unique weak solution~$f$ of \eqref{Doi_eq} such that $f(0)=f_0$. 
This solution is global in time (the definition~\ref{definition_weak_solution} is valid for any time $T>0$). Moreover, $f$ is a classical solution, belonging to $C^\infty ((0,+\infty)\times \mathbb{S})$ (and even analytic in space, see Appendix \ref{analyticity}).
\end{thm}
\begin{remark}
In this case, we do not have any uniform bound on $H^s(\mathbb{S})$, and we can derive the same existence theorem for the case $\sigma =0$ (see Remark \ref{no_noise}), but only for the case $s\geqslant -\frac{n-1}2$ (which does not include all radon signed-measures).

Another remark is that if we change the sign in front of the alignment term in Doi equation \eqref{Doi_eq} (taking $K(\omega,\bar{\omega})=\omega \cdot \bar{\omega}$, every particle tends to go away from the mean direction), then we can derive a conservation relation in the same way. 
But here the “dissipation term” is~$\widetilde{\mathcal D}(f)=2\sigma \|f-1\|^2_{\widetilde H^{-\frac{n-3}2}} + \frac2{(n-2)!}|J[f]|^2\geqslant 2\sigma (n-1)\mathcal{H}(f)$, without any condition on~$\sigma >0$. 
So in any Sobolev space $H^s(\mathbb{S})$, with $s>-\frac{n-1}2$ we have global exponential decay of the solution to the uniform distribution, with rate~$(n-1)\sigma $.
\end{remark}

\subsection{Study of the supercritical case $\sigma <\frac1n$}
In this section, we fix $\sigma <\frac1n$ and we study the behavior of a solution as $t\rightarrow +\infty $.
We will write $\kappa $ for $\kappa (\sigma)$ and $c$ for $c(\kappa (\sigma))$.
We first establish that the limit set of equilibria~$\mathcal E_\infty $ given by LaSalle principle (Proposition \ref{lasalle}) depends only on the fact that $J[f_0]$ is zero or not. 

\begin{prop}
\label{distinction_J_zero}
If $J[f_0]=0$ then $\mathcal E_\infty $ is reduced to the uniform distribution. Equation \eqref{Doi_eq} becomes the heat equation. 
We have exponential decay to the uniform distribution with rate $2n\sigma $ in any $H^s(\mathbb{S})$.

If $J[f_0]\neq 0$ then $J[f(t)]\neq 0$ for all $t>0$. The limit set $\mathcal E_\infty =\{M_{\kappa \Omega}, \Omega \in \mathbb{S} \}$ consists in all the non-isotropic equilibria. Furthermore, we have for any $s\in \mathbb{R}$,
\begin{equation}
\lim_{t\rightarrow \infty}\|f(t)-M_{\kappa \Omega (t)}\|_{H^s}=0,
\end{equation}
where $\Omega (t)=\frac{J[f(t)]}{|J[f(t)]|}$ is the mean direction of $f(t)$.
\end{prop}

\begin{proof}
First of all, we write the equation for $J[f]$, multiplying equation \eqref{Doi_eq} and integrating on the sphere. We get
\begin{align}
\frac{\mathrm d}{\mathrm dt} J[f] &= -\sigma (n-1)J[f] + \left(\int_\mathbb{S}(\mathrm{Id}- \omega \otimes \omega)\, f \,\mathrm d\omega \right) J[f]\nonumber\\
\label{ODEJ} &=\left((1-(n-1)\sigma)\mathrm{Id} - \int_\mathbb{S}\omega \otimes \omega \, f\right) J[f],
\end{align}
which can be viewed as a first order linear ODE of the form $\frac{\mathrm d}{\mathrm dt} J[f]=M(t)J[f]$. 
The matrix~$M$ is a smooth function of time, so we have a global unique solution. 
Consequently, if $J[f(t_0)]=0$ for $t_0\geqslant 0$, then we have $J[f(t)]=0$, for all $t\geqslant 0$, 
and equation~\eqref{Doi_eq} reduces to the heat equation. 
The distribution $f$ has no component on the first eigenspace of the Laplace-Beltrami operator, 
and the second eigenvalue is~$2n$, so we have exponential decay with rate $2n\sigma $ in any $H^s$ norm. 

Now we suppose that $J[f_0]\neq 0$, so by the previous argument we have $J[f(t)]\neq 0$ for all $t\geqslant 0$. 
There are two possibilities for the limiting set, either the uniform distribution, or the set $\{M_{\kappa \Omega}, \Omega \in \mathbb{S} \}$ (by Proposition \ref{minimum_free_energy}, they do not have the same level of free energy).

In the first case, by LaSalle principle, $f(t)$ converges to the uniform distribution. 
Then the matrix $M(t)=(1-(n-1)\sigma)\mathrm{Id} - \int_\mathbb{S}\omega \otimes \omega \, f$ converges to $(n-1)(\frac1n-\sigma)\mathrm{Id}$. 
Using the ODE for $J[f]$, we get 
\begin{equation*}
\frac12\frac{\mathrm d}{\mathrm dt} |J[f]|^2=J[f]\cdot M(t)J[f]\geqslant ((n-1)(\tfrac1n-\sigma)-\varepsilon)|J[f]|^2,
\end{equation*}
for $t$ sufficiently large. 
Taking $\varepsilon $ sufficiently small, we get that $|J[f]|$ tends to infinity, which is a contradiction.

So we have that $\mathcal E_\infty =\{M_{\kappa \Omega}, \Omega \in \mathbb{S} \}$.
Now suppose that $\|f(t)-M_{\kappa \Omega (t)}\|_{H^s}$ does not tend to $0$.
We take $t_n$ tending to infinity such \mbox{that $\|f(t_n)-M_{\kappa \Omega (t_n)}\|_{H^s}\geqslant \varepsilon >0$.} 
By our LaSalle principle, there exists $\Omega_n\in \mathbb{S}$ such that $\|f(t_n)-M_{\kappa \Omega_n}\|_{H^s}\rightarrow 0$. 
Up to extracting, we can suppose that $\Omega_n\rightarrow \Omega_\infty \in \mathbb{S}$, so $f(t_n)\rightarrow M_{\kappa \Omega_\infty}$ in $H^s(\mathbb{S})$. 
In particular we have that $J[f(t_n)]\rightarrow c(\kappa)\Omega_\infty $, and then $\Omega (t_n)\rightarrow \Omega_\infty $. 
Then $M_{\kappa \Omega (t_n)}$ converges to $M_{\kappa \Omega_\infty}$, giving that~$\|f(t_n)-M_{\kappa \Omega (t_n)}\|_{H^s}\rightarrow 0$, which is a contradiction.
\end{proof}

Now we focus on the case $J[f_0]\neq 0$. We define $\Omega (t)$ as in the previous proposition, and we will expand the solution around $M_{\kappa \Omega (t)}$. 
We first show the convergence in~$L^2(\mathbb{S})$ to a given equilibrium, with exponential rate, assuming conditions on the initial data.

\begin{prop}
\label{asymptotic_rate}
There exists an “asymptotic rate” $r_\infty (\sigma)>0$ satisfying the following property.

Suppose that $\|f(t)-M_{\kappa \Omega (t)}\|_{H^s}$ is uniformly bounded on $[t_0,+\infty)$ by a constant~$K$, with $s>\frac{3(n-1)}2$.
Then for all $r<r_\infty (\sigma)$, there exists $\Omega_\infty \in \mathbb{S}$ and $\delta,C>0$, such that if~$\|f(t_0)-M_{\kappa \Omega (t_0)}\|_{L^2}\leqslant \delta $, we have 
\begin{equation*}
\|f(t)-M_{\kappa \Omega_\infty}\|_{L^2} \leqslant C \|f(t_0)-M_{\kappa \Omega (t_0)}\|_{L^2} e^{-r(t-t_0)}.
\end{equation*}
The constants $\delta $ and $C$ depend only on $\sigma $, $s$, $K$, and $r$.
Moreover, as $\sigma \rightarrow \frac1n$, we have that~$r_\infty (\sigma)\geqslant 2(n-1)(\frac1n-\sigma)+O((\frac1n-\sigma)^{\frac32})$.
\end{prop}

\begin{proof}
We first introduce some notations. 
When there is no confusion, we just write~$\Omega $ for $\Omega (t)$, and we will always assume $t\geqslant t_0$. 
We write $\cos\theta =\omega \cdot \Omega $.
We denote by $\langle \cdot \rangle_{M_{\kappa \Omega}}$ the mean of a function against the probability measure $M_{\kappa \Omega}$.

We have the following identities (we recall, by Lemma \ref{beta_positive}, that $\beta =c^2+n\sigma -1$ is positive): 
\begin{gather}
\langle \omega \rangle_{M_{\kappa \Omega}}=\langle \cos\theta \rangle_{M_{\kappa \Omega}}\Omega =c\,\Omega, \nonumber\\
\label{cos2}\langle \cos^2\theta \rangle_{M_{\kappa \Omega}}=1-(n-1)\sigma, \nonumber\\
\label{cos2bis}\langle (\cos\theta -c)^2\rangle_{M_{\kappa \Omega}}=1-(n-1)\sigma -c^2=\sigma -\beta >0.
\end{gather}
We can write $f=(1+ h)M_{\kappa \Omega}$, then we have $\langle h\rangle_{M_{\kappa \Omega}}=0$. 
Since $\Omega $ is the direction of~$J[f]=\langle (1+h)\omega \rangle_{M_{\kappa \Omega}}$, we get that $\langle h\omega \rangle_{M_{\kappa \Omega}}=\langle h\cos\theta \rangle_{M_{\kappa \Omega}}\Omega $.

So we can do an expansion of the free energy and its dissipation in terms of~$h$.
Since we know that $M_{\kappa \Omega (t)}$ is a critical point of $\mathcal F$, we already know that the expansion of~$\mathcal F((1+h)M_{\kappa \Omega})-\mathcal F(M_{\kappa \Omega})$ will contain no term of order $0$ and $1$ in $h$. 
We get, using~\eqref{def_F},
\begin{equation*}
\mathcal F((1+h)M_{\kappa \Omega})-\mathcal F(M_{\kappa \Omega}) = \sigma \tfrac12\langle h^2\rangle_{M_{\kappa \Omega}}-\tfrac12|\langle h\omega \rangle_{M_{\kappa \Omega}}|^2 + O(\|h\|_{\infty}^3).
\end{equation*}
Using Sobolev embedding and interpolation, we have (writing $C$ for a generic constant, depending only on $\sigma $, $s$, and $K$)
\begin{equation*}
\|f-M_{\kappa \Omega}\|_\infty \leqslant C\|f-M_{\kappa \Omega}\|_{H^{\frac{n-1}2}}\leqslant C\|f-M_{\kappa \Omega}\|^{1-\frac{n-1}{2s}}_{L^2}K^{\frac{n-1}{2s}}.
\end{equation*}
So since $1-\frac{n-1}{2s}>\frac23$ and $f-M_{\kappa \Omega}=hM_{\kappa \Omega}$, with $M_{\kappa \Omega}$ uniformly bounded below and above, we get that $\|h\|_{\infty}^3=o(\langle h^2\rangle_{M_{\kappa \Omega}})$ (and more precisely, for any $\varepsilon >0$ there exists $\eta >0$ depending only on $\varepsilon $, $\sigma $, $s$, and $K$ such that $\|h\|_{\infty}^3\leqslant \varepsilon \langle h^2\rangle_{M_{\kappa \Omega}}$ as soon as~$\langle h^2\rangle_{M_{\kappa \Omega}}\leqslant \eta $).
We get
\begin{equation}
\mathcal F(f)-\mathcal F(M_{\kappa \Omega}) = \tfrac12[\sigma \langle h^2\rangle_{M_{\kappa \Omega}}-\langle h\cos\theta \rangle^2_{M_{\kappa \Omega}}] + o(\langle h^2\rangle_{M_{\kappa \Omega}}).
\label{Festimate}
\end{equation}
We use the definition \eqref{def_D} of $\mathcal D(f)$:
\begin{align*}
D(f)&=\langle (1+h)|\nabla (\sigma \ln(M_{\kappa \Omega}(1+h))-\langle (1+h)\omega \rangle_{M_{\kappa \Omega}}\cdot \omega)|^2\rangle_{M_{\kappa \Omega}}\\
&=\langle (1+h)|\nabla (\sigma \ln(1+h)-\langle h\cos\theta \rangle_{M_{\kappa \Omega}}\cos\theta)|^2\rangle_{M_{\kappa \Omega}}\\
&\geqslant (1-\|h\|_{\infty})\langle |\nabla (\sigma \ln(1+h)-\langle h\cos\theta \rangle_{M_{\kappa \Omega}}\cos\theta)|^2\rangle_{M_{\kappa \Omega}}.
\end{align*}
Now we can derive a Poincaré inequality of the form
\begin{equation*}
\langle |\nabla g|^2\rangle_{M_{\kappa \Omega}}\geqslant \Lambda_\kappa \langle (g-\langle g\rangle_{M_{\kappa \Omega}})^2\rangle_{M_{\kappa \Omega}}.
\end{equation*}
Indeed, we use the fact that $M_{\kappa \Omega}$ is positive and bounded:
\begin{align}
\langle |\nabla g|^2\rangle_{M_{\kappa \Omega}}&\geqslant \min M_{\kappa \Omega}\int_\mathbb{S} |\nabla g|^2\nonumber\\
&\geqslant \min M_{\kappa \Omega}(n-1)\int_\mathbb{S} (g-\textstyle\int_\mathbb{S} g)^2\nonumber\\\\
&\geqslant \tfrac{\min M_{\kappa \Omega}}{\max M_{\kappa \Omega}}(n-1)\langle (g-\textstyle\int_\mathbb{S} g)^2\rangle_{M_{\kappa \Omega}}\nonumber\\\\
&\geqslant (n-1)e^{-2\kappa}\langle (g-\langle g\rangle_{M_{\kappa \Omega}})^2\rangle_{M_{\kappa \Omega}}.
\label{minmax_poincare}
\end{align}
Actually this is a rough estimate, we have here $\Lambda_\kappa \geqslant (n-1)e^{-2\kappa}$, a more precise study of $\Lambda_\kappa $ could be done using separation of variable. 
The problem then reduces to finding the smallest eigenvalue of a one-dimensional Sturm-Liouville problem, but even in that case, we did not manage to find a better estimate for now.

So we finally get
\begin{align*}
\mathcal D(f)&\geqslant (1-\|h\|_{\infty})\Lambda_\kappa \langle [\sigma \ln(1+h)-\sigma \langle \ln(1+h)\rangle_{M_{\kappa \Omega}}\hspace{-4pt} - \langle h\cos\theta \rangle_{M_{\kappa \Omega}}(\cos\theta -c)]^2\rangle_{M_{\kappa \Omega}}\\
&\geqslant (1-\|h\|_{\infty})\Lambda_\kappa \langle [\sigma h - \langle h\cos\theta \rangle_{M_{\kappa \Omega}}(\cos\theta -c)+O(\|h\|_{\infty}^2)]^2\rangle_{M_{\kappa \Omega}}\\
&\geqslant (1-\|h\|_{\infty})\Lambda_\kappa (\sigma^2\langle h^2\rangle_{M_{\kappa \Omega}}-(\beta +\sigma) \langle h\cos\theta \rangle_{M_{\kappa \Omega}}^2)+O(\|h\|_{\infty}^3).
\end{align*}
With the same argument as before, we get that
\begin{equation}
\mathcal D(f)\geqslant \Lambda_\kappa (\sigma^2\langle h^2\rangle_{M_{\kappa \Omega}}-(\beta +\sigma) \langle h\cos\theta \rangle_{M_{\kappa \Omega}}^2) + o(\langle h^2\rangle_{M_{\kappa \Omega}}).
\label{Dboundedbelow}
\end{equation}
The goal is now to express the bounds in \eqref{Dboundedbelow} and \eqref{Festimate} as the sum of positive terms.
Indeed, we expect to have a Grönwall’s inequality which will give a rate of convergence.

We set $\alpha =\frac1{\sigma -\beta}\langle h\cos\theta \rangle_{M_{\kappa \Omega}}$, and we write $h=\alpha (\cos\theta -c) + g$. Using \eqref{cos2bis} we have that $\alpha $ is well defined since $\sigma -\beta >0$ and we get~$\langle g\rangle_{M_{\kappa \Omega}}=\langle g\omega \rangle_{M_{\kappa \Omega}}=0$.

Plugging $\langle h^2\rangle_{M_{\kappa \Omega}}=(\sigma -\beta)\alpha^2+\langle g^2\rangle_{M_{\kappa \Omega}}$ into \eqref{Festimate} and \eqref{Dboundedbelow} gives
\begin{gather}
\label{Falphag}\mathcal F(f)-\mathcal F(M_{\kappa \Omega}) = \tfrac12 [\beta (\sigma -\beta)\alpha^2 + \sigma \langle g^2\rangle_{M_{\kappa \Omega}}]+o(\langle h^2\rangle_{M_{\kappa \Omega}}),\\
\begin{split}\mathcal D(f)&\geqslant \Lambda_\kappa (\beta^2(\sigma -\beta) \alpha^2+\sigma^2\langle g^2\rangle_{M_{\kappa \Omega}})+o(\langle h^2\rangle_{M_{\kappa \Omega}})\\
&\geqslant \Lambda_\kappa \beta (\beta (\sigma -\beta) \alpha^2+\sigma \langle g^2\rangle_{M_{\kappa \Omega}})+o(\langle h^2\rangle_{M_{\kappa \Omega}}).
\end{split}\nonumber
\end{gather}
So for all $r<\Lambda_\kappa \beta $, if $\langle h^2\rangle_{M_{\kappa \Omega}}$ is sufficiently small, we have $\mathcal D(f)\geqslant r (\mathcal F(f)-\mathcal F(M_{\kappa \Omega}))$. 
Using the conservation relation \eqref{conservation_free_energy}, there exists $\delta_0>0$ (depending only on $\sigma $, $s$, $K$ and $r$) such that if $\|f(t)-M_{\kappa \Omega (t)}\|_{L^2}\leqslant \delta_0$, we have
\begin{equation*}
\frac{\mathrm d}{\mathrm dt}[\mathcal F(f)-\mathcal F(M_{\kappa \Omega})]=-\mathcal D(f)\leqslant -2r[\mathcal F(f)-\mathcal F(M_{\kappa \Omega})].
\end{equation*}
Then we obtain, for all $T$, such that $\|f-M_{\kappa \Omega}\|_{L^2}\leqslant \delta_0$ on $[t_0,T]$,
\begin{equation*}
\mathcal F(f(T))-\mathcal F(M_{\kappa \Omega (T)})\leqslant [\mathcal F(f(t_0))-\mathcal F(M_{\kappa \Omega (t_0)})] e^{-2r(T-t_0)},
\end{equation*}
and then, using the estimate \eqref{Falphag}, we get that for $t\in [t_0,T]$,
\begin{equation}
\|f-M_{\kappa \Omega}\|_{L^2}\leqslant C_0\|f(t_0)-M_{\kappa \Omega (t_0)}\|_{L^2} e^{-r(t-t_0)}.
\label{finalgronwall}
\end{equation}
So if we take $\delta <\frac{\delta_0}{C_0}\leqslant \delta_0$, and we start with $\|f(t_0)-M_{\kappa \Omega (t_0)}\|_{L^2}\leqslant \delta $, we get that~$\|f-M_{\kappa \Omega}\|_{L^2}\leqslant \delta_0$ on $[t_0,T]$ for all $T\geqslant t_0$. 
Otherwise, the largest of such a $T$ would satisfy $\delta_0=\|f(T)-M_{\kappa \Omega (T)}\|_{L^2}\leqslant C\delta e^{-r(T-t_0)}<\delta_0$.
So the inequality \eqref{finalgronwall} holds for all $t\in [t_0,+\infty)$.

It remains to prove that $\Omega (t)$ converges to some $\Omega_\infty $, if we want to have strong convergence to a given steady state.
This is possible using the ODE satisfied by $\Omega $.

Indeed, we have $J[f]=c\Omega +\langle h\omega \rangle_{M_{\kappa \Omega}}=(c+\alpha (\sigma -\beta))\Omega $, and then 
\begin{equation*}
\frac{\mathrm d}{\mathrm dt}J[f]=(c+\alpha (\sigma -\beta))\frac{\mathrm d}{\mathrm dt}\Omega +(\sigma -\beta)\Omega \frac{\mathrm d}{\mathrm dt}\alpha.
\end{equation*}
So applying $\mathrm{Id}-\Omega \otimes \Omega $ to the ODE \eqref{ODEJ} gives an ODE for $\Omega $, in terms of $\alpha $ and $g$.
We get 
\begin{align*}
 (\mathrm{Id}-\Omega \otimes \Omega)\frac{\mathrm d}{\mathrm dt} J[f] &= - (\mathrm{Id}-\Omega \otimes \Omega)\left(\int_\mathbb{S}\omega \otimes \omega \, f \,\mathrm d\omega \right) J[f]\\
&=-(c+\alpha (\sigma -\beta))(\mathrm{Id}-\Omega \otimes \Omega)[\langle h\cos\theta \,\omega \rangle_{M_{\kappa \Omega}}+\langle \cos\theta \,\omega \rangle_{M_{\kappa \Omega}}].
\end{align*}
Since $\langle (\cos\theta -c)\cos\theta \,\omega \rangle_{M_{\kappa \Omega}}$ and $\langle \cos\theta \,\omega \rangle_{M_{\kappa \Omega}}$ are parallel to $\Omega $, we get that
\begin{equation*}
(c+\alpha (\sigma -\beta))\frac{\mathrm d\Omega}{\mathrm dt}=-(c+\alpha (\sigma -\beta))(\mathrm{Id}-\Omega \otimes \Omega)\langle g\cos\theta \,\omega \rangle_{M_{\kappa \Omega}}.
\end{equation*}
Since $(c+\alpha (\sigma -\beta))$ is the norm of $J[f]$, it is never zero, and we get (the notation~$C$ standing for a generic constant depending only on $r$, $s$, $\sigma $ and $K$)
\begin{equation*}
\left|\frac{\mathrm d\Omega}{\mathrm dt}\right|\leqslant C\sqrt{\langle g^2\rangle_{M_{\kappa \Omega}}}\leqslant C\|f-M_{\kappa \Omega}\|_{L^2}.
\end{equation*}
So we have exponential decay of $\frac{\mathrm d\Omega}{\mathrm dt}$ with rate $r$, in particular $\Omega $ is converging to some $\Omega_\infty \in \mathbb{S}$. 
More precisely,
\begin{equation*}
|\Omega (t)-\Omega_\infty |\leqslant \int_t^\infty |\tfrac{\mathrm d\Omega}{\mathrm dt}|\mathrm dt\leqslant C\|f(t_0)-M_{\kappa \Omega (t_0)}\|_{L^2} e^{-r(t-t_0)}.
\end{equation*}
Now we have that $\|M_{\kappa \Omega (t)}-M_{\kappa \Omega_\infty}\|_{L^2}\leqslant C|\Omega (t)-\Omega_\infty |$ (the function $\Omega \mapsto e^{\kappa \omega \cdot \Omega}$ from~$\mathbb{S}$ to $\mathbb{R}$ is globally Lipschitz with a constant independent of $\omega \in \mathbb{S}$).
So we get the final estimation: 
\begin{equation*}
\|f-M_{\kappa \Omega_\infty}\|_{L^2}\leqslant \|f-M_{\kappa \Omega}\|_{L^2}+\|M_{\kappa \Omega (t)}-M_{\kappa \Omega_\infty}\|_{L^2}\leqslant C\|f(t_0)-M_{\kappa \Omega (t_0)}\|_{L^2} e^{-r(t-t_0)}.
\end{equation*}
So the proposition is true with $r_\infty (\sigma)=\Lambda_\kappa \beta >0$. 
By the estimate \eqref{minmax_poincare}, we know that $\Lambda_\kappa \geqslant (n-1)e^{-2\kappa}$. 
And by the expansions of $c$ and $\kappa $ as $\sigma \rightarrow \frac1n$ given in \eqref{expansion_coefficients}, we get that $r_\infty (\sigma)\geqslant 2(n-1)(\frac1n-\sigma)+O((\frac1n-\sigma)^{\frac32})$.
\end{proof}

By Proposition \ref{distinction_J_zero}, we have that $f(t)-M_{\kappa \Omega (t)}$ tends to zero in any $H^s(\mathbb{S})$.
So the hypotheses of Proposition \ref{asymptotic_rate}, for any $r<r_\infty (\sigma)$, are satisfied for some~$t_0>0$.

Once more, by interpolation and uniform boundedness on $[t_0,+\infty)$ of the $H^p$ norm, we have
\begin{align*}
\|f-M_{\kappa \Omega_\infty}\|_{H^s}&\leqslant C\|f-M_{\kappa \Omega_\infty}\|^{1-\frac{s}{p}}_{L^2}\|f-M_{\kappa \Omega_\infty}\|_{H^p}^{\frac{s}{p}}\\
&\leqslant \widetilde{C}\|f(t_0-M_{\kappa \Omega (t_0)}\|^{1-\frac{s}{p}}_{L^2}e^{-r(1-\frac{s}{p})(t-t_0)},
\end{align*}
so taking $p$ sufficiently large, we also get exponential convergence for the $H^s$ norm, with rate $r(1-\delta)$ for any $\delta >0$.

Finally we have that for all $r<r_\infty (\sigma)$ and $s$, there exists some time $t_0$ and $C>0$ such that~$\|f-M_{\kappa \Omega_\infty}\|_{H^s}\leqslant Ce^{-rt}$ for $t\geqslant t_0$.
We can even get rid of the constant $C$ since for any $\widetilde r<r$ and $t$ sufficiently large $Ce^{-rt}\leqslant e^{-\widetilde rt}$.

\subsection{Study of the critical case $\sigma =\frac1n$}
For any $\sigma \in (0,+\infty)\setminus\{\frac1n\}$, we have exponential convergence to some equilibrium. 
However the rate of convergence tends to $0$ when $\sigma $ is close to $\frac1n$ (in the case where~$J[f_0]\neq 0$).
So we do not expect to have a similar rate of convergence in the critical case.

First of all, we know by Proposition \ref{minimum_free_energy} that the solution converges (in any~$H^s(\mathbb{S})$) to the uniform distribution as time goes to infinity. 
The goal of this section is to estimate the speed of convergence to this equilibrium.

\begin{prop}
\label{algebraic_decay}
Suppose that $\|f(t)-1\|_{H^s}$ is uniformly bounded on $[t_0,+\infty)$ by a constant $K$, with $s>\frac{7(n-1)}2$.

Then for all $C>1$, there exists $\delta >0$, such that if $\|f(t_0)-1\|_{L^2}\leqslant \delta $, we have, for~$t\geqslant t_0$,
\begin{equation*}
\|f(t)-1\|_{L^2} \leqslant \frac{C}{\sqrt{\frac1{\sqrt{2(n+2)}\|f(t_0)-1\|_{L^2}}+\frac{2(n-1)}{n(n+2)}(t-t_0)}}.
\end{equation*}
The constant $\delta $ depends only on $\sigma $, $s$, $K$, and $C$.
\end{prop}

\begin{proof}
As in the previous section, we work on $[t_0,+\infty)$.
We write $f=1+h$ and as in the previous case, we suppose that $J[f_0]\neq 0$. 
By the same argument used in Proposition~\ref{distinction_J_zero}, we have that $J[f(t)]\neq 0$ for all $t>0$, so we define $\Omega (t)$ as the unit vector $\frac{J[f(t)]}{|J[f(t)]|}$. 
Similarly we denote $\langle \cdot \rangle $ for the mean of a function on the unit sphere and $\cos\theta $ for $\omega \cdot \Omega $. 

We have $\langle h\rangle =0$. Since $\Omega $ is the direction of $J[f]=\langle (1+h)\omega \rangle =\langle h\omega \rangle $, we get that~$\langle h\omega \rangle =\langle h\cos\theta \rangle \Omega $.

We perform an expansion of the free energy and its dissipation in terms of~$h$.
We get, using \eqref{def_F} and taking $\sigma =\frac1n$,
\begin{equation*}
\mathcal F(1+h)= \tfrac1n(\tfrac12\langle h^2\rangle -\tfrac16\langle h^3\rangle +\tfrac1{12}\langle h^4\rangle)-\tfrac12\langle h\cos\theta \rangle^2 + O(\|h\|_{\infty}^5).
\end{equation*}
Now we write $\alpha =n\langle h\cos\theta \rangle $ and we define 
\begin{equation}
g=h-\alpha \cos\theta -\tfrac12\alpha^2(\cos^2\theta -\tfrac1n)-\tfrac16\alpha^3(\cos^3\theta -\tfrac3{n+2}\cos\theta).
\label{def_g_critical}
\end{equation}
We have $\langle \cos^4\theta \rangle =\frac3{n(n+2)}$ (we have used the formula \eqref{induction_ap} to compute $\frac{4!a_2}{a_0}=\langle \cos^4\theta \rangle $). 
Since we have $\langle \cos^3\theta \rangle =\langle \cos\theta \rangle =0$, and $\langle \cos^2\theta \rangle =\frac1n$, we get $\langle g\rangle =\langle g\cos\theta \rangle =0$.
We will see that the terms of order $2$ in $g$ will not vanish in the expansion of the free energy and the dissipation term.
But we will need to expand the free energy in~$\alpha $ up to order $4$, and the dissipation term up to order $6$ in $\alpha$.
We have
\begin{align}
\label{h_L2_critical}
\tfrac12\langle h^2\rangle &=\tfrac12\langle g^2\rangle +\tfrac1{2n}\alpha^2+\tfrac{n-1}{4n^2(n+2)}\alpha^4+\tfrac1{2}\alpha^2\langle g\cos^2\theta \rangle +O(\alpha^3\|g\|_\infty +\alpha^5),\\
-\tfrac16\langle h^3\rangle &=-\tfrac{n-1}{2n^2(n+2)}\alpha^4-\tfrac1{2}\alpha^2\langle g\cos^2\theta \rangle +O(\|g\|_\infty^3+\alpha \|g\|^2_\infty +\alpha^3\|g\|_\infty +\alpha^5),\nonumber\\
\tfrac1{12}\langle h^4\rangle &=\tfrac1{4n(n+2)}\alpha^4+O(\|g\|_\infty^4+\alpha \|g\|_\infty^3+\alpha^2\|g\|^2_\infty +\alpha^3\|g\|_\infty +\alpha^5).\nonumber
\end{align}
We finally get
\begin{equation}
\mathcal F(1+h)= \tfrac1{2n}\langle g^2\rangle + \tfrac1{4n^3(n+2)}\alpha^4 + O(\|g\|_\infty^3+\alpha \|g\|^2_\infty +\alpha^3\|g\|_\infty +\alpha^5).
\label{expansion_F_critical}
\end{equation}
Using the inequality $a^pb^q\leqslant sa^{\frac ps}+(1-s)b^{\frac q{1-s}}$ for $s\in (0,1)$, with $a=\alpha $ and $b=\|g\|_\infty $, we get that $\alpha \|g\|_\infty^2\leqslant \frac15\alpha^5+\frac45\|g\|_\infty^{2+\frac12}$ and $\alpha^3\|g\|_\infty \leqslant \frac35\alpha^5+\frac25\|g\|_\infty^{2+\frac12}$. 

By Sobolev embedding and interpolation, as in the previous section, we have 
\begin{equation}
\|g\|_\infty \leqslant C\|g\|_{L^2}^{1-\frac{n-1}{2s}}\|g\|^{\frac{n-1}{2s}}_{H^s},
\label{interpolation_critical}
\end{equation}
with $1-\frac{n-1}{2s}>\frac67$. 

Since $\alpha $ is controlled by $\|h\|_{H^s}$, using the definition \eqref{def_g_critical} of $g$, we have a bound \mbox{for~$\|g\|_{H^s}$ on $[t_0,+\infty)$,} depending only on $s$ and $K$. 
We finally get $\|g\|^{2+\frac12}_\infty \leqslant C\langle g^2\rangle^\mu $, \mbox{with~$\mu >\frac12(2+\frac12)\frac67>1$.}
 
So using \eqref{h_L2_critical} and \eqref{expansion_F_critical}, we get that for any $\varepsilon >0$, there exists $\delta >0$ such if~$\|h\|_{L^2}\leqslant \delta $, we have
\begin{gather}
(1-\varepsilon)(\langle g^2\rangle +\tfrac1n \alpha^2)\leqslant \langle h^2\rangle \leqslant (1+\varepsilon)(\langle g^2\rangle +\tfrac1n \alpha^2)\nonumber\\
\label{F_estimate_critical}
(1-\varepsilon)(\tfrac1{2n}\langle g^2\rangle + \tfrac1{4n^3(n+2)}\alpha^4)\leqslant \mathcal F(1+h)\leqslant \tfrac{1+\varepsilon}{4n^3(n+2)}(2n^2(n+2)\langle g^2\rangle + \alpha^4).
\end{gather}
From that, up to take a smaller $\delta $, we obtain
\begin{equation}
\label{Fh}
\tfrac{1-\varepsilon}{1+\varepsilon}\,2n\mathcal F(1+h)\leqslant \langle h^2\rangle \leqslant \tfrac{1+\varepsilon}{\sqrt{1-\varepsilon}}\,2\sqrt{n(n+2)\mathcal F(1+h)}.
\end{equation}
We now estimate the dissipation term. 
We use the definition \eqref{def_D} of $\mathcal D(f)$ and the Poincaré inequality to get:
\begin{align}
\mathcal D(f)&=\langle (1+h)|\nabla (\tfrac1n\ln(1+h)-\langle (1+h)\omega \rangle \cdot \omega)|^2\rangle \nonumber\\
&=\langle (1+h)|\nabla (\tfrac1n\ln(1+h)-\langle h\cos\theta \rangle \cos\theta)|^2\rangle \nonumber\\
&\geqslant \tfrac{n-1}{n^2}(1-\|h\|_{\infty})\langle [\underbrace{\ln(1+h)-\langle \ln(1+h)\rangle -n\langle h\cos\theta \rangle \cos\theta}_{\mathcal S(h)}]^2\rangle.
\label{expansion_D_critical}
\end{align}
We have 
\begin{align*}
\mathcal S(h)&=\ln(1+h)-\langle \ln(1+h)\rangle -n\langle h\cos\theta \rangle \cos\theta \\
&=h-\langle h\rangle -\alpha \cos\theta -\tfrac12(h^2-\langle h^2\rangle)+\tfrac13(h^3-\langle h^3\rangle) +O(\|h\|^4).
\end{align*}
We compute,
\begin{align*}
h-\langle h\rangle -\alpha \cos\theta &=g + \frac12\alpha^2(\cos^2\theta -\tfrac1n)+\tfrac16\alpha^3(\cos^3\theta -\tfrac3{n+2}\cos\theta)\\
-\tfrac12(h^2-\langle h^2\rangle)&=-\tfrac12(\alpha^2+\alpha^3\cos\theta)(\cos^2\theta -\tfrac1n)+O(\|g\|^2+\alpha \|g\|_\infty +\alpha^4)\\
\tfrac13(h^3-\langle h^3\rangle)&=\tfrac13\alpha^3\cos^3\theta +O(\|g\|_\infty^3+\alpha \|g\|_\infty^2+\alpha^2\|g\|_\infty +\alpha^4).
\end{align*}
So
\begin{align}
\langle \mathcal S(h)^2\rangle &=\langle [g+\tfrac16\alpha^3(\tfrac3n-\tfrac3{n+2})\cos\theta)]^2\rangle +O(\|g\|^3+\alpha \|g\|^2+\alpha^4\|g\|_\infty +\alpha^7)\nonumber\\
&=\langle g^2\rangle +\tfrac1{n^3(n+2)^2}\alpha^6+O(\|g\|_\infty^3+\alpha \|g\|_\infty^2+\alpha^4\|g\|_\infty +\alpha^7).
\label{expansion_s}
\end{align}
As before, we get that $\alpha \|g\|_\infty^2\leqslant \frac17\alpha^7+\frac67\|g\|_\infty^{2+\frac13}$ and $\alpha^4\|g\|_\infty \leqslant \frac47\alpha^7+\frac37\|g\|_\infty^{2+\frac13}$. 
\mbox{Using~\eqref{interpolation_critical},} we get $\|g\|^{2+\frac13}_\infty \leqslant C\langle g^2\rangle^\mu $, with $\mu >\frac12(2+\frac13)\frac67=1$.
So using \eqref{expansion_D_critical} and~\eqref{expansion_s}, up to take a smaller $\delta $, we have, for $\|h\|_{L^2}\leqslant \delta $,
\begin{equation*}
\mathcal D(f) \geqslant (1-\varepsilon)\tfrac{n-1}{n^2}(\langle g^2\rangle +\tfrac1{n^3(n+2)^2}\alpha^6).
\end{equation*}
Now for any $C,C'>0$, if we take $\alpha $ and $g$ sufficiently small (so again up to take a smaller $\delta $), we have that $C\langle g^2\rangle +\alpha^6\geqslant (C'\langle g^2\rangle +\alpha^4)^{\frac32}$. So we get 
\begin{equation*}
\mathcal D(f) \geqslant (1-\varepsilon)\tfrac{n-1}{n^5(n+2)^2}(2n^2(n+2)\langle g^2\rangle +\alpha^4)^{\frac32}.
\end{equation*}
Putting this together with \eqref{F_estimate_critical} and the conservation relation \eqref{conservation_free_energy}, we get that for any~$0<\varepsilon <1$, there exists $\delta_0>0$ such, as soon as $\|h\|_{L^2}\leqslant \delta_0$, we have
\begin{equation*}
\frac{\mathrm d}{\mathrm dt}\mathcal F(f)=-\mathcal D(f)\leqslant -\frac{8(n-1)(1-\varepsilon)}{(1+\varepsilon)^{\frac32}\sqrt{n(n+2)}}[\mathcal F(f)]^{\frac32}.
\end{equation*}
Then we obtain, for all $T$ such that $\|h\|_{L^2}\leqslant \delta_0$ on $[t_0,T]$,
\begin{equation}
\mathcal F(f(T))^{-\frac12}\geqslant \mathcal F(f(t_0))^{-\frac12}+\tfrac{4(n-1)(1-\varepsilon)}{(1+\varepsilon)^{\frac32}\sqrt{n(n+2)}}(t-t_0).
\label{gronwall_critical}
\end{equation}
Then, using \eqref{Fh}, we get that for $t\in [t_0,T]$,
\begin{equation*}
\|h\|^{-2}_{L^2}\geqslant \textstyle\frac{\sqrt{1-\varepsilon}}{(1+\varepsilon)\,2\sqrt{n(n+2)}}[\sqrt{\frac{2n(1-\varepsilon)}{1+\varepsilon}}\|h(t_0)\|_{L^2}^{-1}+\tfrac{4(n-1)(1-\varepsilon)}{(1+\varepsilon)^{\frac32}\sqrt{n(n+2)}}(t-t_0)].
\end{equation*}
We write $C=\frac{(1+\varepsilon)^{\frac54}}{(1-\varepsilon)^{\frac34}}$ (a one-to-one correspondence between $0<\varepsilon <1$ and~$C>1$) and we get 
\begin{equation}
\|h\|_{L^2}\leqslant C\left[\tfrac1{\sqrt{2(n+2)}\|h(t_0)\|_{L^2}}+\tfrac{2(n-1)}{n(n+2)}(t-t_0)\right]^{-\frac12}.
\label{finalgronwall_critical}
\end{equation}
So if we take $\delta <\min(\delta_0,\frac1{C^2\sqrt{2(n+2)}}\,\delta_0^2)$, and $\|h(t_0)\|_{L^2}\leqslant \delta $, we get that $\|h\|_{L^2}\leqslant \delta_0$ on~$[t_0,T]$ for all~$T\geqslant t_0$. 
Otherwise, the largest of such a $T$ would satisfy 
\begin{equation*}
\delta_0=\|h(T)\|_{L^2}\leqslant C\left[\tfrac1{\sqrt{2(n+2)}\delta}\right]^{-\frac12}<\delta_0.
\end{equation*}
So the inequality \eqref{finalgronwall_critical} holds for all $t\in [t_0,+\infty)$, which ends the proof.
\end{proof}

With this proposition, since $f$ tends to the uniform distribution in any $H^s(\mathbb{S})$, we get that for any~$r<\frac{2(n-1)}{n(n+2)}$, there exists $t_0$ such that we have $\|f(t)-1\|_{L^2}\leqslant \frac1{\sqrt{r(t-t_0)}}$, for~$t\geqslant t_0$.
We can even get rid of the $t_0$ in this inequality since for any $r<\widetilde r<\frac{2(n-1)}{n(n+2)}$, for~$t$ sufficiently large, we have $\frac1{\sqrt{\widetilde r(t-t_0)}}\leqslant \frac1{\sqrt{rt}}$.

As in the previous section, using interpolation to deal with the other Sobolev norms of the solution would lead, for any~$\eta >0$ and $t$ sufficiently large, to an inequality of the form~$\|f(t)-1\|_{H^p}\leqslant C_\eta t^{-\frac12+\eta}$.
But we can actually do slightly better.
Indeed we have, following the notations of the proof and using \eqref{def_g_critical},

\begin{equation*}
\|h\|_{H^s}\leqslant |\alpha |\|\cos\theta \|_{H^p}+C_2\alpha^2+C_3|\alpha |^3+\|g\|_{H^p}.
\end{equation*}
We have $\|\cos\theta \|_{H^p}=(n-1)^{\frac p2}$.
We take $t_0>0$ satisfying the conditions of the proposition and such that $\|h\|_{L^2}\leqslant \delta $.
We have that $g$ is uniformly bounded in any~$H^p(\mathbb{S})$, and so by interpolation, we have $\|g\|_{H^s}\leqslant C_\eta \|g\|^{1-\eta}_{L^2}$ for any $\eta >0$. Now using~\eqref{gronwall_critical} and \eqref{F_estimate_critical}, we get

\begin{equation*}
(\tfrac1{2n}\langle g^2\rangle + \tfrac1{4n^3(n+2)}\alpha^4)^{-\frac12}\geqslant \tfrac{4(n-1)(1-\varepsilon)^{\frac32}}{(1+\varepsilon)^{\frac32}\sqrt{n(n+2)}}(t-t_0),
\end{equation*}
which gives $\|g\|_{L^2}=O(t^{-1})$ and $\alpha^2\leqslant \tfrac{(1+\varepsilon)^{\frac32}n(n+2)}{2(n-1)(1-\varepsilon)^{\frac32}(t-t_0)}$.
So finally, for any $\eta >0$, we have that~$\|h\|_{H^p}\leqslant (n-1)^{\frac p2}\sqrt{\tfrac{(1+\varepsilon)^{\frac32}n(n+2)}{2(n-1)(1-\varepsilon)^{\frac32}(t-t_0)}}+O(t^{-1+\eta})$. 
This gives that there exists~$t_1\geqslant t_0$ such that for all $t\geqslant t_1$, we have $\|h\|_{H^p}\leqslant (1+\varepsilon)(n-1)^{\frac p2}\sqrt{\tfrac{(1+\varepsilon)^{\frac32}n(n+2)}{2(n-1)(1-\varepsilon)^{\frac32}(t-t_0)}}$. 

This is true for any $\varepsilon >0$. In conclusion, we have that for any $r<\frac{2}{n(n-1)^{p-1}(n+2)}$, there exists $t_1$ such that for $t\geqslant t_1$, we have $\|f(t)-1\|_{H^p}\leqslant \frac1{\sqrt{rt}}$.

\subsection{Summary}
In summary we can state the following theorem:

\begin{thm}
Convergence to equilibrium.

Suppose $f_0$ is a probability measure, belonging to $H^s(\mathbb{S})$ (this is always the case for some $s<-\frac{n-1}2$).

Then there exists a unique weak solution $f$ to Doi equation \eqref{Doi_eq}, satisfying the initial condition $f(0)=f_0$.

Furthermore, this is a classical solution, positive for all time $t>0$, and belonging to $C^\infty ((0,+\infty)\times \mathbb{S})$.

If $J[f_0]\neq 0$, then we have the three following cases, depending on $\sigma $.
\begin{itemize}
\item If $\sigma >\frac1n$, then $f$ converges exponentially fast to the uniform distribution, with global rate~$(n-1)(\sigma -\frac1n)$ in any $H^p$ norm. 

More precisely, for all $t_0>0$, there exists a constant $C>0$ depending only on~$t_0, s, p, n$, and $\sigma $, such that for all $t\geqslant t_0$, we have
\begin{equation*}
\|f(t)-1\|_{H^{p}}\leqslant C\|f_0\|_{H^{s}}e^{-(n-1)(\sigma -\frac1n)t}.
\end{equation*}
\item If $\sigma <\frac1n$, then there exists $\Omega \in \mathbb{S}$ such that $f$ converges exponentially fast to $M_{\kappa \Omega}$, with asymptotic rate $r_\infty (\sigma)>0$ in any $H^p$ norm. 

More precisely, for all $r<r_\infty (\sigma)$, there exists $t_0>0$ (depending on $f_0$) such that for all $t>t_0$, we have 
\begin{equation*}
\|f(t)-M_{\kappa \Omega}\|_{H^{p}}\leqslant e^{-rt}.
\end{equation*} 

When $\sigma $ is close to $\frac1n$ we have that $r_\infty (\sigma)\sim2(n-1)(\frac1n-\sigma)$.

\item If $\sigma =\frac1n$, then $f$ converges to the uniform distribution in any $H^p$ norm, with asymptotic rate $\sqrt{\frac{n(n-1)^{p-1}(n+2)}{2t}}$.

More precisely, for all $r<\frac2{n(n-1)^{p-1}(n+2)}$, there exists $t_0>0$ (depending on $f_0$) such that for all $t>t_0$, we have
\begin{equation*}
\|f(t)-1\|_{H^{p}}\leqslant \frac{1}{\sqrt{rt}}.
\end{equation*}
\end{itemize}

If $J[f_0]=0$ the equation reduces to the heat equation on the sphere, so $f$ converges to the uniform distribution, exponentially with global rate $2n\sigma $ in any $H^p$ norm. 

\end{thm}
 
For the subcritical case $\sigma >\frac1n$, we used Theorem \ref{new_entropy}. 
In the case where~$p<-\frac{n-1}2$, a simple embedding gives $\|f(t)-1\|_{H^{p}}\leqslant \|f(t)-1\|_{H^{-\frac{n-1}2}}$ so we only have to show the result for $p\geqslant -\frac{n-1}2$. 
We get
\begin{equation*}
\|f-1\|^2_{H^{p}}\leqslant C\|f(t_0)-1\|_{H^{p}}e^{-(n-1)(\sigma -\tfrac1n)(t-t_0)}\leqslant C\|f(t_0)\|_{H^{p}}e^{-(n-1)(\sigma -\tfrac1n)(t-t_0)}.
\end{equation*}
The last inequality comes from the fact that $f(t_0)$ is a probability density function, so~$f(t_0)-1$ is the orthogonal projection of $f(t_0)$ on the space of mean-zero functions.
Using Proposition \ref{regularity_boundedness}, we get $\|f(t_0)\|_{H^{p}}\leqslant C_{t_0}\|f_0\|_{H^{s}}$ in the case $p\geqslant s$. 
Otherwise we just use a simple embedding to get first $\|f(t_0)\|_{H^{p}}\leqslant \|f(t_0)\|_{H^{s}}$ and then by the same proposition $\|f(t_0)\|_{H^{p}}\leqslant C\|f_0\|_{H^{s}}$.

Then the results in the case $\sigma <\frac1n$ and $\sigma =\frac1n$ are a summary of the conclusions of the two previous subsections.
However, although it gives a clear understanding of how fast the solution converges to the equilibrium, in some sense, this summary is not as accurate as Propositions \ref{asymptotic_rate} and \ref{algebraic_decay}, which give a kind of stability: starting close to an equilibrium, the solution stays close. 

\section{Conclusion}
In this paper, we have investigated all the possible dynamics in large time for the Doi-Onsager equation \eqref{Doi_eq} with dipolar potential. 
We have obtained a rate of convergence towards the equilibrium given any initial condition and any noise parameter $\sigma >0$, for all dimension $n\geqslant 2$. 

The rate of convergence to the anisotropic steady state, in the case $\sigma <\frac1n$, depends on a Poincaré constant which does not seem easy to estimate. 
A better knowledge of the behavior of this constant, for example as the noise parameter $\sigma $ tends to zero, would be useful to understand the limiting case $\sigma =0$, where we have existence and uniqueness of the solution.
In this limit, the steady states are given by the sum of two antipodal Dirac masses $(1-\alpha)\delta_{\Omega}+\alpha \delta_{-\Omega}$ with $\Omega \in \mathbb{S}$ and $0\leqslant \alpha \leqslant \frac12$. We conjecture that if the initial condition is continuous (and with non zero initial momentum), then the solution converges to one of these steady states, with $\alpha =0$.

It should also be possible to get the same kind of rates for the Maier-Saupe potential, but there the classification of the initial conditions leading to a given type of equilibria is much more difficult, in particular in the case where two types of equilibria are stable. 

\section*{Acknowledgements}

The authors would like to thank Pierre Degond for initiating this project, suggesting to work on it, and for many stimulating discussions.\\
The research of J.-G. L. was partially supported by NSF grant DMS 10-11738.
The authors are also thankful for the support of the Mathematical Sciences Center at Tsinghua University.
\appendix
\section{Appendix}
\subsection{Using the spherical harmonics}
\label{appendix_harmonics}

For the following we will use the spherical harmonics, so we recall some preliminaries results. 
We fix $n\geqslant 2$ and work on $\mathbb{R}^n$ and its unit sphere $\mathbb{S}_{n-1}$.

\begin{deft}
\label{deft_sph_harm}
A spherical harmonic of degree $\ell $ on $\mathbb{S}_{n-1}$ is the restriction to~$\mathbb{S}_{n-1}$ of a homogeneous polynomial of degree $\ell $ in $n$ variables (seen as a function $\mathbb{R}^n\rightarrow \mathbb{R}$) which is an harmonic function (a function $P$ such that $\Delta P=0$, where $\Delta $ is the usual Laplace operator in $\mathbb{R}^n$). 
We denote $\mathcal{H}^{(n)}_\ell $ the set of spherical harmonics of degree~$\ell $ on~$\mathbb{S}_{n-1}$ (including~$0$ so they are vector spaces).
\end{deft}
We know that the space of homogeneous polynomials of degree $\ell $ in $n$ variables has dimension $\binom{n+\ell -1}{n-1}$ (the number of $n$-tuples $(i_1,\dots i_n)$ of sum $\ell $). 
Writing an arbitrary homogeneous polynomial $P$ of degree $\ell $ under the form $P=\sum_{i=0}^\ell Q_{\ell -i}X_n^{i}$, with the polynomials $Q_i$ being homogeneous of degree $i$ in the first $n-1$ variables, and imposing that $P$ is an harmonic function gives the following conditions (taking the term in $X_n^{i-2}$), for $i\in \llbracket 0,\ell -2\rrbracket $: $\Delta Q_{\ell -i}+(i+1)(i+2)Q_{\ell -i-2}=0$. 
Finally the polynomial~$P$ is only determined by the polynomials $Q_\ell $ and $Q_{\ell -1}$ in $n-1$ variables, of respective degrees $\ell $ and~$\ell -1$. This gives the dimension of the space of spherical harmonics.

\begin{prop}
The dimension of $\mathcal{H}^{(n)}_\ell $ is given by
\begin{equation*}
k^{(n)}_\ell = \tbinom{n+\ell -2}{n-2}+ \tbinom{n+\ell -3}{n-2}=\tbinom{n+\ell -1}{n-1}- \tbinom{n+\ell -3}{n-1}.
\end{equation*} 
\end{prop}

The second expression comes from two successive applications of Pascal’s triangle rule, and will be useful in the following.
It can also be seen by the following property\footnote{This can be shown using the appropriate inner product $(P,Q)\mapsto P(\mathrm D)Q$ on the space of homogeneous polynomials $P$ of degree $\ell $, where $P(\mathrm D)$ is defined as~$\frac{\partial^\ell}{\partial_{X_1}^{\alpha_1}\dots \partial_{X_n}^{\alpha_n}}$ if~$P=X_1^{\alpha_1}\dots X_n^{\alpha_n}$, and extended by linearity (so for example, we have that $|X|^2(\mathrm D)=\Delta $).
If we denote by $E$ the space of polynomials of the form $P=|X|^2Q$, with $Q$ of degree $\ell -2$, then the orthogonal of $E$ consists in all the polynomials $P$ such that for all $Q$ of degree $\ell -2$, we have $(|X|^2Q)(\mathrm D)P=Q(\mathrm D)\Delta P=0$, that is to say in all the polynomials $P$ such that $\Delta P=0$. So the claimed decomposition is just the orthogonal decomposition, on $E$ and $E^\perp $.}: every homogeneous polynomial $P$ of degree $\ell $ can be decomposed in a unique way as~$H+|X|^2Q$, where $H$ is harmonic of degree $\ell $ and $Q$ is homogeneous of degree~$\ell -2$. 
Iterating this decomposition, we get 
\begin{equation*}
P=H_\ell +|X|^2H_{\ell -2}+|X|^4H_{\ell -4}+\dots +
\begin{cases}|X|^\ell H_{0} & \ell \text{ even}\\|X|^{\ell -1}H_{1} & \ell \text{ odd}
\end{cases},
\end{equation*}
where the polynomials $H_i$ are harmonic of degree $i$. This shows that any restriction of a polynomial on the sphere is equal to a sum of spherical harmonics (the terms $|X|^{2i}$ are constant when restricted to the sphere). This gives, with the Stone-Weierstrass theorem, that the sum of spherical harmonics are dense in $L^2(\mathbb{S}_{n-1})$ (since they are dense in the continuous functions). Together with the radial decomposition of the Laplacian $\Delta =\frac1{r^{n-1}}\partial_r(r^{n-1}\partial_r)+\frac1{r^2}\Delta_\omega $ (where $\Delta_\omega $ is the Laplace Beltrami operator on the sphere $\mathbb{S}_{n-1}$, which is self-adjoint in $L^2(\mathbb{S}_{n-1})$), we get the following result:
\begin{prop}
The spaces $\mathcal{H}^{(n)}_\ell $, for $\ell \in \mathbb{N}$, are the eigenspaces of the Laplace Beltrami operator $\Delta_\omega $ on the sphere $\mathbb{S}_{n-1}$ for the eigenvalues $-\ell (\ell +n-2)$.
They are pairwise orthogonal and complete in $L^2(\mathbb{S}_{n-1})$. 
\end{prop}

We can construct a basis of $\mathcal{H}^{(n)}_\ell $ by induction on the dimension, using the separation of variables.
We describe this construction and will use it in the following.

For a given unit vector $e_{n}\in \mathbb{R}^{n}$, we take an orthonormal basis $(e_1,\dots,e_{n})$ of~$\mathbb{R}^{n}$.
Any~$\omega \in \mathbb{S}_{n-1}\setminus\{e_{n},-e_{n}\}$ can be written $\omega =\cos\theta e_{n} + \sin\theta v$, with $\theta \in (0,\pi)$ and~$v\in \mathbb{S}_{n-2}$.
We identify $\mathbb{R}^{n-1}$ with the vector space spanned by $(e_1,\dots,e_{n-1})$. 
The special case $n=2$ works if we consider $S_0=\{e_1,-e_1\}$. 

By convention, the only spherical harmonics on $S_0$ are the constant functions (of degree~$0$) and the functions $e_1\mapsto c$, $-e_1\mapsto -c$ (of degree $1$).

Now, for $n\geqslant 1$, we choose an orthonormal basis $(Z_m^1,\dots Z_m^{k^{(n-1)}_m})$ of $\mathcal{H}^{(n-1)}_m$ for any $m\in \mathbb{N}$ and we have the following result:

\begin{prop}
\label{induction_basis_SH}
There exists polynomials $Q_{\ell,m}$ of degree $\ell -m$ such that if we denote~$Y_{\ell,m}^k(\omega)=Q_{\ell,m}(\cos\theta)\sin^m\theta Z_m^k(v)$, then the $Y_{\ell,m}^k$ for $m\in \llbracket 0,\ell \rrbracket, k\in \llbracket 1,k^{(n-1)}_m\rrbracket $ form an orthonormal basis of $H_\ell^{(n)}$.
\end{prop}
\begin{proof}
Writing $Y_{\ell,m}^k(\omega)=Q_{\ell,m}(\cos\theta)\sin^m\theta Z_m^k(v)$ and asking it to be a spherical harmonic is equivalent to the following linear ODE for $Q_{\ell,m}$ (we recall that the Laplace-Beltrami operator is given by $\sin^{2-n}\theta \partial_\theta (\sin^{n-2}\theta \partial_\theta) + \frac1{\sin^2\theta}\Delta_v$ in this coordinates):
\begin{equation*}
\begin{split}\sin^{2-n}\partial_\theta (&-\sin^{n+m-1}\theta Q'_{\ell,m}(\cos\theta) + m \cos\theta \sin^{n+m-3}\theta Q_{\ell,m}(\cos\theta))\\
&-m(m+n-3)Q_{\ell,m}(\cos\theta)\sin^{m-2}\theta =-\ell (\ell +n-2)Q_{\ell,m}(\cos\theta)\sin^{m}\theta.
\end{split}
\end{equation*}
We write $x=\cos\theta $ and this equation transforms into
\begin{equation*}
(1-x^2)Q''_{\ell,m}-(n+2m-1)xQ'_{\ell,m} + (\ell -m)(\ell +n+m-2)Q_{\ell,m}=0.
\end{equation*}
This equation is a particular form of the Jacobi differential equation, where the two parameters $\alpha $ and $\beta $ are equal (also called Gegenbauer differential equation).
One solution of this differential equation is a polynomial, called ultraspherical polynomial (a particular case of the Jacobi Polynomials, also called Gegenbauer polynomials), and denoted $P_i^{(\lambda)}$ following the notation of Szegö in \cite{szego1975orthogonal}. Precisely, it satisfies the differential equation
\begin{equation*}
(1-x^2)y''-(2\lambda +1)xy' + i(i+2\lambda)y=0.
\end{equation*}
Taking $\lambda =m-1+\tfrac{n}2$ and $i=\ell -m$, we get a solution $Q_{\ell,m}= \alpha_{\ell,m} P_{\ell -m}^{(m-1+\tfrac{n}2)}$, where~$\alpha_{\ell,m}$ is a positive constant of normalization, such that $Y_{\ell,m}^k$ is of norm $1$ in~$L^2(\mathbb{S}_{n-1})$. We have to be careful here because $P_i^{(\lambda)}$ is not defined for $\lambda =0$, and so the only special case is~$n=2$, $m=0$, for which we have a solution $Q_{\ell,0}=\sqrt2T_\ell $, where~$T_\ell (\cos\theta)=\cos{\ell \theta}$ (the Chebyshev polynomial of first order of degree $\ell $).

So for a fixed $\ell $, we have constructed a family of spherical harmonics $Y_{\ell,m}^k$ of degree $\ell $ for $m\in \llbracket 0,\ell \rrbracket, k\in \llbracket 1,k^{(n-1)}_m\rrbracket $. They are pairwise orthogonal in $L^2(\mathbb{S}_{n-1})$ since the $Z_{m}^k$ are pairwise orthogonal in $L^2(\mathbb{S}_{n-2})$. The size of this family is exactly 
\begin{equation}
\sum_{m=0}^\ell k^{(n-1)}_m=\sum_{m=0}^\ell \tbinom{n+m-2}{n-2}-\tbinom{n+m-4}{n-2}=\tbinom{n+\ell -2}{n-2}+\tbinom{n+\ell -3}{n-2}=k^{(n)}_\ell,
\end{equation}
which is the dimension of $H_\ell^{(n)}$, so we get that the $Y_{\ell,m}^k$ for $m\in \llbracket 0,\ell \rrbracket, k\in \llbracket 1,k^{(n-1)}_m\rrbracket $ form an orthonormal basis of $H_\ell^{(n)}$.
\end{proof}
From now on, we will use the construction done in the proof.
We have that, for a fixed $m\geqslant 0$, the polynomials $Q_{\ell,m}$ for $\ell \geqslant m$ are a family of orthogonal polynomials for the inner product $(P,Q)\mapsto \int_{-1}^1 P(x)Q(x)(1-x^2)^{m-1+\tfrac{n-1}2}\mathrm dx$. 

We will use three properties on the Gegenbauer polynomials (see~\cite{szego1975orthogonal}) for the following, for $i\geqslant 0$, $\lambda \neq 0$, and $\lambda >-\frac12$ (with the convention~$P_{-1}^{(\lambda)}=0$): 

\begin{gather}
\int_{-1}^1(P_i^{(\lambda)}(x))^2(1-x^2)^{\lambda -\frac12}\mathrm dx = \frac{2^{1-2\lambda}\pi \Gamma (i+2\lambda)}{(i+\lambda)\Gamma^2(\lambda)\Gamma (i+1)}\label{norm_gegenbauer}\\
(i+1) P_{i+1}^{(\lambda)}= 2(i+\lambda)XP_i^{(\lambda)} -(i+2\lambda -1)P_{i-1}^{(\lambda)}\label{induction_gegenbauer}\\
(1-X^2)(P_i^{(\lambda)})'= \frac{1}{2(i+\lambda)}\left((i+2\lambda -1)(i+2\lambda)P_{i-1}^{(\lambda)}-i(i+1) P_{i+1}^{(\lambda)}\right)\label{deriv_gegenbauer}
\end{gather}

We have the following normalization for the $Q_{\ell,m}$: 
\begin{equation*}
\int_{-1}^1 Q^2_{\ell,m}(x)(1-x^2)^{m-1+\frac{n-1}2}\mathrm dx= \int_{-1}^1(1-x^2)^{\frac{n-1}2-1}\mathrm dx.
\end{equation*}

This gives the following relation, together with \eqref{norm_gegenbauer}:
\begin{equation}
\alpha_{\ell +1,m}^2=\tfrac{\left(\ell +\frac{n}2\right)(\ell +1-m)}{\left(\ell +\frac{n}2-1\right)(\ell +m+n-2)}\alpha_{\ell,m}^2.\label{induction_alpha}
\end{equation}

By the previous construction, we can decompose $g={\sum_{k,\ell,m}} c_{\ell,m}^kY_{\ell,m}^k$ and we have~$\int_{\mathbb{S}_{n-1}}g^2=\sum_{k,\ell,m} |c_{\ell,m}^k|^2$. 
Since $g$ is of mean zero, we have $c_{0,0}^1=0$ (the only spherical harmonic of degree $0$ is the constant function $1$). 
So from now, the indices~$k,\ell,m$ of the sum will mean $\ell >0, m\in \llbracket 0,\ell \rrbracket, k\in \llbracket 1,k^{(n-1)}_m\rrbracket $.

We decompose in the same way $h={\sum_{k,\ell,m}} d_{\ell,m}^kY_{\ell,m}^k$.
We give a first formula, in the form of a lemma.

\begin{lemma}
\label{g_nabla_h}
We have
\begin{equation}
e_{n}\cdot \int_{\mathbb{S}_{n-1}}g\nabla h =\frac12\sum_{k,\ell,m}b_{\ell,m}[(\ell +n-1)c_{\ell,m}^kd_{\ell +1,m}^k-\ell c_{\ell +1,m}^kd_{\ell,m}^k],
\end{equation}
where $b_{\ell,m}=\frac{\sqrt{\ell -m+1}\sqrt{\ell +m+n-2}}{\sqrt{\ell +\frac{n}2-1}\sqrt{\ell +\frac{n}2}}\leqslant 1$.
\end{lemma}
\begin{proof}
We have 
\begin{equation*}
e_{n}\cdot \nabla Y_{\ell,m}^k=-\sin\theta \partial_\theta Y_{\ell,m}^k=\left[(1-X^2)Q'_{\ell,m}-mXQ_{\ell,m}\right](\cos\theta)\sin^m\theta Z_m^k(v),
\end{equation*}
and using the inductions formulas \eqref{induction_gegenbauer}, \eqref{deriv_gegenbauer} and \eqref{induction_alpha}, we get
\begin{equation}
(1-X^2)Q'_{\ell,m}-mXQ_{\ell,m}=\frac12[b_{\ell -1,m}(\ell +n-2)Q_{\ell -1,m}-b_{\ell,m}\ell Q_{\ell +1,m}],
\label{induction_q}
\end{equation}
where $b_{\ell,m}$ is given in the statement of the lemma.
In the special case $n=2,m=0$, using the formula $Q_{\ell,0}(\cos\theta)=\cos{\ell \theta}$ gives the same formula as \eqref{induction_q}, with $b_{\ell,0}=1$.

So we have that $\int_{\mathbb{S}_{n-1}}e_{n}\cdot \nabla Y_{\ell,m}^k Y_{\ell ',m'}^{k'}$ can be non-zero only if~$m=m'$, $k=k'$, and~$\ell =\ell '\pm 1$. By bilinearity, together with the fact that $Y_{\ell,m}^k$ form an orthonormal basis, this gives the claimed formula.
\end{proof}
Now we have all the tools to prove Lemma \ref{main_lemma} (we recall it here).
\setcounter{lemma}{0}
\begin{lemma}
Estimates on the sphere.
\begin{enumerate}
\item If $h$ in $\dot{H}^{-s+1}(\mathbb{S})$ and $g$ in $\dot{H}^{s}(\mathbb{S})$, the following integral is well defined and we have
\begin{equation*}
\left|\int_\mathbb{S} g\nabla h\right|\leqslant C\|g\|_{\dot{H}^s}\|h\|_{\dot{H}^{-s+1}}
\end{equation*}
where the constant depends only on $s$ and $n$.
\item
We have the following estimation, for any $g \in \dot{H}^{s+1}(\mathbb{S})$:
\begin{equation*}
\left|\int_{\mathbb{S}}g \nabla (-\Delta)^sg\right|\leqslant C\|g\|^2_{\dot{H}^{s}},
\end{equation*}
where the constant depends only on $s$ and $n$.

\item We have the following identity, for any $g \in \dot{H}^{-\frac{n-3}2}$:
\begin{equation*}
\int_{\mathbb{S}}g \nabla \widetilde{\Delta}_{n-1}^{-1}g=0
\end{equation*}
\end{enumerate}
\end{lemma}
\begin{proof}

Using Lemma~\ref{g_nabla_h}, we get
\begin{align*}
e_{n}\cdot \int_{\mathbb{S}_{n-1}}g\nabla h \leqslant &\frac12 \sum_{k,\ell,m}\sqrt{\tfrac{\ell +n-1}{\ell +1}}\left(\tfrac{\lambda_{\ell +1}}{\lambda_\ell}\right)^{\frac s2}|\lambda_\ell^{\frac{s}2}c_{\ell,m}^k||\lambda_{\ell +1}^{\frac{-s+1}2}d_{\ell +1,m}^k|\\
&+\frac12 \sum_{k,\ell,m}\sqrt{\tfrac{\ell}{\ell +n-2}}\left(\tfrac{\lambda_\ell}{\lambda_{\ell +1}}\right)^{\frac s2}|\lambda_{\ell +1}^{\frac{s}2}c_{\ell +1,m}^k||\lambda_\ell^{\frac{-s+1}2}d_{\ell,m}^k|\\
\leqslant &C\|g\|_{\dot{H}^s}\|h\|_{\dot{H}^{-s+1}}
\end{align*}
where $\lambda_{\ell}=\ell (\ell +n-2)$ (the eigenvalue of $-\Delta $ for the spherical harmonics of degree~$\ell $). 
The last line comes from the fact that the sequences $\sqrt{\frac{\ell +n-1}{\ell +1}}\left(\frac{\lambda_{\ell +1}}{\lambda_\ell}\right)^{\frac s2}$ and~$\sqrt{\tfrac{\ell}{\ell +n-2}}\left(\tfrac{\lambda_\ell}{\lambda_{\ell +1}}\right)^{\frac s2}$ are bounded (they tend to $1$), together with a Cauchy-Schwarz inequality.
This gives the first part of the lemma, since this is true for any unit vector $e_n$.

Now we take $h=(-\Delta)^sg$, which is replacing $d_{\ell,m}^k$ by $\lambda_\ell^sc_{\ell,m}^k$ in Lemma~\ref{g_nabla_h}. We get
\begin{align*}
e_{n}\cdot \int_{\mathbb{S}_{n-1}}g\nabla (-\Delta)^sg &=\sum_{k,\ell,m}\frac12b_{\ell,m}c_{\ell +1,m}^kc_{\ell,m}^k[(\ell +n-1)\lambda^s_{\ell +1}-\ell \lambda^s_{\ell}]\\
&\leqslant \sum_{k,\ell,m}|\lambda_{\ell +1}^{\frac{s}2}c_{\ell +1,m}^k||\lambda_\ell^{\frac{s}2}c_{\ell,m}^k||(\ell +n-1)\left(\tfrac{\lambda_{\ell +1}}{\lambda_{\ell}}\right)^{\frac s2}-\ell \left(\tfrac{\lambda_\ell}{\lambda_{\ell +1}}\right)^{\frac s2}|\\
&\leqslant C\|g\|^2_{\dot{H}^{s}}.
\end{align*}
Indeed we have that $\frac{\lambda_{\ell +1}}{\lambda_{\ell}}=1-\frac2{\ell}+O(\ell^{-2})$, so $|(\ell +n-1)\left(\tfrac{\lambda_{\ell +1}}{\lambda_{\ell}}\right)^{\frac s2}-\ell \left(\tfrac{\lambda_\ell}{\lambda_{\ell +1}}\right)^{\frac s2}|$ is bounded (it tends to $(n-1)+2s$).
Since this computation is now valid for any unit vector $e_n$, this gives the second part of the lemma.

The last part is straightforward by taking $h=\widetilde{\Delta}^{-1}_{n-1}g$ with Lemma~\ref{g_nabla_h}. 
According to the definition given in \eqref{spectral_conformal_laplacian}, we have $d_{\ell,m}^k=\frac1{\ell (\ell +1)\dots (\ell +n-2)}\,c_{\ell,m}^k$. 
We get
\begin{equation*}
e_{n}\cdot \int_{\mathbb{S}_{n-1}}g\nabla \widetilde{\Delta}^{-1}g =\sum_{k,\ell,m}\frac12b_{\ell,m}c_{\ell +1,m}^kc_{\ell,m}^k[\tfrac{\ell +n-1}{(\ell +1)\dots (\ell +n-1)}-\tfrac{\ell}{\ell (\ell +1)\dots (\ell +n-2)}]=0,
\end{equation*}
which is true for any unit vector $e_n$.
\end{proof}

\subsection{Analyticity of the solution}
\label{analyticity}

Following \cite{constantin2005dissipativity}, we will show that the solution belongs to a special Gevrey class.
We define the space $G_r$, as the set of functions $g$ (with mean zero) such that $\widetilde{\Delta}^{-\frac12}_{n-1}e^{r(-\Delta)^{\frac12}}g$ is in $L^2(\mathbb{S})$. 
Using the notations of the previous proof, this is an Hilbert space associated to the inner product
\begin{equation*}
\langle g,h\rangle_{\dot G_r^{s}}=\sum_{k,\ell,m}\frac{e^{2r\sqrt{\ell (\ell +n-2)}}}{\ell (\ell +1)\dots (\ell +n-2)}c_{\ell,m}^kd_{\ell,m}^k.
\end{equation*}
The norm on this Hilbert space will be written $\|\cdot \|_{G_r}$.

We take $r$ a function of $t$, we will denote its time derivative by $\dot{r}$.
For a given solution $f=1+g$, we put $h=\widetilde{\Delta}^{-1}_{n-1}e^{2r(-\Delta)^{\frac12}}g$ in \eqref{weak_solution_meanzero}.

The left-hand side is
\begin{align*}
\langle \partial_tg&,\widetilde{\Delta}^{-1}_{n-1}e^{2r(-\Delta)^{\frac12}}g\rangle =\sum_{k,\ell,m}\frac{e^{2r\sqrt{\ell (\ell +n-2)}}}{\ell (\ell +1)\dots (\ell +n-2)}c_{\ell,m}^k\frac{\mathrm d}{\mathrm dt}c_{\ell,m}^k\\
=&\sum_{k,\ell,m}\frac12\frac{\mathrm d}{\mathrm dt}(\tfrac{e^{2r\sqrt{\ell (\ell +n-2)}}}{\ell (\ell +1)\dots (\ell +n-2)}|c_{\ell,m}^k|^2) - \dot{r}\tfrac{e^{2r\sqrt{\ell (\ell +n-2)}}}{\sqrt{\ell}(\ell +1)\dots (\ell +n-1)\sqrt{\ell +n-2}}|c_{\ell,m}^k|^2\\
=&\frac12\frac{\mathrm d}{\mathrm dt}\|g\|^2_{G_r}- \dot{r} \|(-\Delta)^{\frac14}g\|^2_{G_r}.
\end{align*}
Using Lemma~\ref{g_nabla_h}, we get
\begin{align*}
e_n\cdot \langle &g,\nabla \widetilde{\Delta}^{-1}_{n-1}e^{2r(-\Delta)^{\frac12}}g\rangle =\frac12\sum_{k,\ell,m}b_{\ell,m}c_{\ell +1,m}^kc_{\ell,m}^k\frac{e^{2r\sqrt{(\ell +1)(\ell +n-1)}}-e^{2r\sqrt{\ell (\ell +n-2)}}}{(\ell +1)\dots (\ell +n-2)}\\
\leqslant &\frac12\sum_{k,\ell,m}\tfrac{\sqrt[4]{(\ell +1)(\ell +n-1)}e^{r\sqrt{(\ell +1)(\ell +n-1)}}}{\sqrt{(\ell +1)\dots (\ell +n-2)(\ell +n-1)}}|c_{\ell +1,m}^k|\tfrac{\sqrt[4]{\ell (\ell +n-2)}e^{r\sqrt{\ell (\ell +n-2)}}}{\sqrt{\ell (\ell +1)\dots (\ell +n-2)}}|c_{\ell,m}^k| \\
&\times \sqrt[4]{\tfrac{\ell (\ell +n-1)}{(\ell +1)(\ell +n-2)}}\left(e^{r\left(\sqrt{(\ell +1)(\ell +n-1)}-\sqrt{\ell (\ell +n-2)}\right)}-e^{-r\left(\sqrt{(\ell +1)(\ell +n-1)}-\sqrt{\ell (\ell +n-2)}\right)}\right)\\
\leqslant & \sinh(r(\sqrt{2n}-\sqrt{n-1})) \|(-\Delta)^{\frac14}g\|^2_{G_r}.
\end{align*}
Indeed the expression $\sqrt{(\ell +1)(\ell +n-1)}-\sqrt{\ell (\ell +n-2)}$ is a decreasing function of~$\ell \geqslant 0$. 
Since this is valid for any unit vector $e_n$, we get 
\begin{equation*}
\left|J[g]\cdot \langle g,\nabla \widetilde{\Delta}^{-1}_{n-1}e^{2r(-\Delta)^{\frac12}}g\rangle \right|\leqslant \sinh(r(\sqrt{2n}-\sqrt{n-1})) \|(-\Delta)^{\frac14}g\|^2_{G_r}.
\end{equation*}

Now since $\|(-\Delta)^{\frac14}g\|^2_{G_r}\leqslant \frac1{\sqrt{n-1}}\|(-\Delta)^{\frac12}g\|^2_{G_r}$, and $|J[h]|\leqslant \frac{e^{2r\sqrt{n-1}}}{(n-1)!}|J[g]|$, we finally get
\begin{equation*}
\frac12\frac{\mathrm d}{\mathrm dt}\|g\|^2_{G_r}+[\sigma -\tfrac1{\sqrt{n-1}}(\dot{r}+\sinh(r(\sqrt{2n}-\sqrt{n-1})))]\, \|(-\Delta)^{\frac12}g\|^2_{G_r}\leqslant \frac{e^{2r\sqrt{n-1}}}{(n-2)!}.
\end{equation*}

As soon as $\dot{r}+\sinh(r(\sqrt{2n}-\sqrt{n-1}))<(\sigma -\varepsilon)\sqrt{n-1}$ and $r$ is bounded in time, we have that $\|g\|^2_{G_r}$ is uniformly bounded (and we can indeed take $r(t)=\delta \min(1,t)$ for~$\delta $ sufficiently small), provided~$g_0$ is in $G_{r(0)}$. 
So if we have $r(0)=0$, we only need~$g_0$ to be in $H^{-\frac{n-1}2}$. 
If it is not the case, by instantaneous regularization (Proposition~\ref{regularity_boundedness}) we have it for any time~$t>0$. 
Since $G_{r}$, for~$r>0$, is a subset of the set of analytical functions on the sphere, we get that any solution becomes instantaneously analytic in space.


\begin{thebibliography}{10}

\bibitem{beckner1993sharp}
W.~Beckner.
\newblock {Sharp Sobolev inequalities on the sphere and the Moser--Trudinger
  inequality}.
\newblock {\em Annals of Mathematics}, 138(1):213--242, 1993.

\bibitem{cao2000gevrey}
C.~Cao, M.A. Rammaha, and E.S. Titi.
\newblock {Gevrey regularity for nonlinear analytic parabolic equations on the
  sphere}.
\newblock {\em Journal of Dynamics and Differential Equations}, 12(2):411--433,
  2000.

\bibitem{constantin2004remarks}
P.~Constantin, I.~Kevrekidis, and E.S. Titi.
\newblock {Remarks on a Smoluchowski equation}.
\newblock {\em Dynamical Systems}, 11(1):101--112, 2004.

\bibitem{constantin2005dissipativity}
P.~Constantin, E.S. Titi, and J.~Vukadinovic.
\newblock {Dissipativity and Gevrey regularity of a Smoluchowski equation}.
\newblock {\em Indiana University Mathematics Journal}, 54(4):949--970, 2005.

\bibitem{degond2011macroscopic}
P.~Degond, A.~Frouvelle, and J.G. Liu.
\newblock {Macroscopic limits and phase transition in a system of
  self-propelled particles}.
\newblock in progress.

\bibitem{degond2008continuum}
P.~Degond and S.~Motsch.
\newblock {Continuum limit of self-driven particles with orientation
  interaction}.
\newblock {\em M3AS}, 18:321--366, 2008.

\bibitem{doi1981molecular}
M.~Doi.
\newblock {Molecular dynamics and rheological properties of concentrated
  solutions of rodlike polymers in isotropic and liquid crystalline phases}.
\newblock {\em Journal of Polymer Science: Polymer Physics Edition},
  19(2):229--243, 1981.

\bibitem{evans1998partial}
L.C. Evans.
\newblock {\em Partial Differential Equations}, volume~19 of {\em Graduate
  Studies in Mathematics}.
\newblock American Mathematical Society, Providence, Rhode Island, 1998.

\bibitem{fatkullin2005critical}
I.~Fatkullin and V.~Slastikov.
\newblock {Critical points of the Onsager functional on a sphere}.
\newblock {\em Nonlinearity}, 18:2565--2580, 2005.

\bibitem{foias1989gevrey}
C.~Foias and R.~Temam.
\newblock {Gevrey class regularity for the solutions of the Navier-Stokes
  equations}.
\newblock {\em Journal of Functional Analysis}, 87(2):359--369, 1989.

\bibitem{frouvelle2011continuous}
A.~Frouvelle.
\newblock {A continuous model for alignment of self-propelled particles with
  anisotropy and density-dependent parameters}.
\newblock preprint.

\bibitem{hsu2002stochastic}
E.~P. Hsu.
\newblock {\em {Stochastic analysis on manifolds}}, volume~38 of {\em Graduate
  Series in Mathematics}.
\newblock American Mathematical Society, Providence, Rhode Island, 2002.

\bibitem{liu2005axial}
H.~Liu, H.~Zhang, and P.~Zhang.
\newblock {Axial symmetry and classification of stationary solutions of
  Doi-Onsager equation on the sphere with Maier-Saupe potential}.
\newblock {\em Commun. Math. Sci}, 3(2):201--218, 2005.

\bibitem{maier1958eine}
W.~Maier and A.~Saupe.
\newblock {Eine einfache molekulare Theorie des nematischen
  kristallinflüssigen Zustandes}.
\newblock {\em Z. Naturforsch.}, 13:564--566, 1958.

\bibitem{mckean1967propagation}
H.P. McKean.
\newblock {Propagation of chaos for a class of non-linear parabolic equations}.
\newblock {\em Lecture Series in Differential Equations}, 7:41--57, 1967.

\bibitem{onsager1949effects}
L.~Onsager.
\newblock {The effects of shape on the interaction of colloidal particles}.
\newblock {\em Annals of the New York Academy of Sciences}, 51(Molecular
  Interaction):627--659, 1949.

\bibitem{szego1975orthogonal}
G.~Szeg{\"o}.
\newblock {\em Orthogonal polynomials}, volume~23 of {\em Colloquium
  Publications}.
\newblock American Mathematical Society, Providence, Rhode Island, 1975.

\bibitem{sznitman1991topics}
A.-S. Sznitman.
\newblock {Topics in propagation of chaos}.
\newblock In {\em École d’Été de Probabilités de Saint-Flour XIX —
  1989}, volume 1464, pages 165--251, Berlin, 1991. Springer.

\bibitem{vicsek1995novel}
T.~Vicsek, A.~Czirok, E.~Ben-Jacob, I.~Cohen, and O.~Shochet.
\newblock {Novel type of phase transition in a system of self-driven
  particles}.
\newblock {\em Physical Review Letters}, 75(6):1226--1229, 1995.

\bibitem{watson1982distributions}
G.~Watson.
\newblock {Distributions on the circle and sphere}.
\newblock {\em Journal of Applied Probability}, 19:265--280, 1982.

\bibitem{zhang2007stable}
H.~Zhang and P.~Zhang.
\newblock {Stable dynamic states at the nematic liquid crystals in weak shear
  flow}.
\newblock {\em Physica D: Nonlinear Phenomena}, 232(2):156--165, 2007.

\bibitem{zhou2005new}
H.~Zhou, H.~Wang, M.G. Forest, and Q.~Wang.
\newblock {A new proof on axisymmetric equilibria of a three-dimensional
  Smoluchowski equation}.
\newblock {\em Nonlinearity}, 18:2815--2825, 2005.

\bibitem{zhou2007characterization}
H.~Zhou, H.~Wang, Q.~Wang, and M.G. Forest.
\newblock {Characterization of stable kinetic equilibria of rigid, dipolar rod
  ensembles for coupled dipole--dipole and Maier--Saupe potentials}.
\newblock {\em Nonlinearity}, 20:277--297, 2007.

\end{thebibliography}
\end{document}